\documentclass{amsart} 

\usepackage{amsmath, amssymb, latexsym, amscd, mathabx}
\usepackage{exscale, cite, eps fig, graphics}

\usepackage[colorlinks]{hyperref}

\renewcommand{\geq}{\geqslant}
\renewcommand{\leq}{\leqslant}

\newtheorem{theorem}{Theorem}[section]
\newtheorem{lemma}[theorem]{Lemma}

\newtheorem*{main-theorem}{Main Theorem}
\newtheorem*{remark*}{Remark}
\numberwithin{equation}{section}

\title[Wave breaking in a shallow water model]
{Wave breaking in a shallow water model}

\author[Hur]{Vera~Mikyoung~Hur}
\address{Department of Mathematics, University of Illinois at Urbana-Champaign, Urbana, IL 61801 USA}
\email{verahur@math.uiuc.edu}

\author[Tao]{Lizheng~Tao}
\address{Department of Mathematics, University of California, Riverside, CA 92521 USA}
\email{ltao@math.ucr.edu}  

\date{\today}


\begin{document}

\maketitle

\begin{abstract}
We prove wave breaking --- bounded solutions with unbounded derivatives --- in the nonlinear nonlocal equations  
which combine the dispersion relation of water waves and the nonlinear shallow water equations, 
and which generalize the Whitham equation to permit bidirectional wave propagation, 
provided that the slope of the initial data is sufficiently negative.
\end{abstract}

\section{Introduction}\label{sec:intro}
As Whitham~\cite{Whitham} emphasized, ``the breaking phenomenon is one of the most intriguing long-standing problems of water wave theory." The {\em nonlinear shallow water equations}:
\begin{equation}\label{E:shallow}
\begin{aligned}
&\partial_t\eta+\partial_x(u(1+a\eta))=0,\\
&\partial_tu+\partial_x\eta+a\,u\partial_xu=0,
\end{aligned}
\end{equation}
approximate the physical problem when the order of the characteristic wavelengh is greater than the undisturbed fluid depth, and they explain {\em wave breaking}. That is, the solution remains bounded but its slope becomes unbounded in finite time. Here $t\in\mathbb{R}$ is proportional to elapsed time, and $x\in\mathbb{R}$ is the spatial variable in the primary direction of wave propagation; $\eta=\eta(x,t)$ represents the free surface displacement from the depth $=1$, and $u=u(x,t)$ is the particle velocity at the rigid horizontal bottom; $a>0$ is the dimensionless nonlinearity parameter. See \cite{Lannes}, for instance, for details. We assume for simplicity that the constant due to gravitational acceleration is $1$.  Note that the phase speed associated with the linear part of \eqref{E:shallow} is independent of the spatial frequency, whereas the speed of a plane wave with the spatial frequency $\xi$ near the quiescent state of water\footnote{The derivation of \eqref{def:cWW} dates back to the work of Airy in 1845!} is 
\begin{equation}\label{def:cWW}
c_{WW}^2(\xi)=\frac{\tanh(\xi)}{\xi}.
\end{equation}
In other words, \eqref{E:shallow} neglects the dispersion effects of the physical problem. 

But the shallow water theory goes too far. It predicts that {\em all} solutions carrying an increase of elevation break. Yet observations have long been established that some waves in water do {\em not} break. Perhaps, the neglected dispersion effects inhibit breaking. 

But, including some\footnote{In the long wave limit as $\xi\to0$, one may expand the right side of \eqref{def:cWW} and find that 
\[
c_{WW}(\xi)=\Big(1-\frac16\xi^2\Big)+O(\xi^4).
\]
} dispersion effects, the {\em Korteweg-de Vries (KdV) equation}:
\begin{equation}\label{E:KdV}
\partial_t\eta+\Big(1+\frac16a\partial_x^2\Big)\partial_x\eta+\frac32a\,\eta\partial_x\eta=0,
\end{equation}
in turn, goes too far and predicts that {\em no} solutions break. 
To conclude, one needs some dispersion effects to satisfactorily explain breaking, but the dispersion of the KdV equation seems too strong. This is not surprising because the phase speed~$=1-\frac16a\xi^2$ associated with the linear part of \eqref{E:KdV} poorly approximates\footnote{A relative error of $10\%$, say, between $c_{WW}(\sqrt{a}\xi)$ and the phase speed for the KdV equation is made for $\sqrt{a}\xi>1.242\dots$.} that of water waves (see \eqref{def:cWW}) when $\xi$ is large. 

Whitham therefore noted that ``it is intriguing to know what kind of simpler mathematical equation (than the governing equations of the water wave problem) could include" the breaking effects, and he put forward (see \cite{Whitham}, for instance)
\begin{equation}\label{E:whitham}
\partial_t\eta+\mathcal{M}_{1/2}\partial_x\eta+\frac32a\,\eta\partial_x\eta=0.
\end{equation}
Here $\mathcal{M}_{1/2}$ is a Fourier multiplier operator, defined via its symbol as
\begin{equation}\label{def:M1/2}
\widehat{\mathcal{M}_{1/2}f}(\xi)=c(\xi)\widehat{f}(\xi),
\end{equation}
and $c=c_{WW}$ (see \eqref{def:cWW}). It combines the dispersion relation of the {\em unidirectional} propagation of water waves and a nonlinearity of the shallow water theory. In a small amplitude and long wavelength regime, where $a=\xi^2\ll1$, the Whitham equation agrees with the KdV equation up to the order of $a$. As a matter of fact, solutions of \eqref{E:whitham}-\eqref{def:M1/2}, where $c=c_{WW}$, and \eqref{E:KdV} exist and they converge to those of the water wave problem up to the order of $a$ during a relevant interval of time; see \cite{Lannes}, for instance, for details. Including the full range of the dispersion in water waves, on the other hand, the Whitham equation may offer an improvement over the KdV equation for short and intermediately long waves. Whitham conjectured that his equation would capture the breaking effects. 

Seliger~\cite{Seliger} made a rather ingenious argument, albeit formal, and claimed that a sufficiently asymmetric solution of \eqref{E:whitham}-\eqref{def:M1/2} breaks, provided that the Fourier transform of $c$ be even, bounded, integrable, and monotonically decay to zero at infinity. Unfortunately, it does not apply to the Whitham equation, because $c_{WW}$ is {\em not} integrable (see \eqref{def:cWW}). Later Constantin and Escher~\cite{CE98} turned Seliger's argument into a rigorous proof. Naumkin and Shishmar\"ev~\cite{NS94} made another breaking argument, provided that the Fourier transform of $c$ and its derivative be integrable and $|c(\xi)|\leq C|\xi|^{-1/3}$ for $|\xi|\gg 1$ for some $C>0$. Unfortunately, the Fourier transform of $c_{WW}$ may not be written explicitly and, hence, the assumptions in \cite{NS94} seem difficult to verify for the Whitham equation. 
While preparing the manuscript, one of the authors~\cite{Hur-breaking} solved Whitham's conjecture. 

In recent years, the Whitham equation gathered renewed attention because of its ability to explain high frequency phenomena in water waves. In particular, one of the authors~\cite{HJ2} proved that a small-amplitude, periodic traveling wave of \eqref{E:whitham}-\eqref{def:M1/2}, where $c=c_{WW}$ (see \eqref{def:cWW}), be spectrally unstable to long wavelength perturbations, provided that the wave number is greater than a critical value, and stable to square integrable perturbations otherwise. In other words, the Whitham equation captures the Benjamin-Feir instability\footnote{A periodic wave train in water is unstable to slow modulations, provided that the carrier wave number times the undisturbed fluid depth is greater than $1.363\dots$; see \cite{BF, Whitham1967}, for instance. } of Stokes waves. By the way, the Benjamin-Feir instability is a high frequency effect, which does not manifest in the KdV and nonlinear shallow water equations.
But the linear operator associated with the Whitham equation does not admit collisions of spectra away from the origin, which numerical computations (see \cite{DO}, for instance) indicate to lead to new kinds of instabilities in the physical problem. To quote Whitham, ``it is intriguing to know what kind of simpler mathematical equation could include" the breaking and other high frequency effects. 

We propose ``bidirectional Whitham" or ``Boussinesq-Whitham" equations:
\begin{equation}\label{E:BW}
\begin{aligned}
&\partial_t\eta+\partial_x(u(1+a\eta))=0,\\
&\partial_tu+\mathcal{M}\partial_x\eta+a\,u\partial_xu=0,
\end{aligned}
\end{equation}
where $\mathcal{M}$ is a Fourier multiplier operator, defined via its symbol as
\begin{equation}\label{def:M}
\widehat{\mathcal{M}f}(\xi)=c^2_{WW}(\xi)\widehat{f}(\xi)
=\frac{\tanh(\xi)}{\xi}\widehat{f}(\xi).
\end{equation}
They combine the dispersion relation of the {\em bidirectional} propagation of water waves and the nonlinear shallow water equations (see \eqref{E:shallow}). The spectrum of the linear operator associated with \eqref{E:BW} is the same as that for the physical problem. In a small amplitude and long wavelength regime, where $a=\xi^2\ll1$, moreover, they agree with a variant\footnote{They do not explicitly appear in the work of Boussinesq. But (280) in \cite{Bnesq1877}, for instance, after several ``higher order terms" drop out, becomes equivalent to \eqref{E:boussinesq}.} of the Boussinesq equations:
\begin{equation}\label{E:boussinesq}
\begin{aligned}
&\partial_t\eta+\partial_x(u(1+a\eta))=0,\\
&\Big(1-\frac13a\partial_x^2\Big)\partial_tu+\partial_x\eta+a\,u\partial_xu=0,
\end{aligned}
\end{equation}
up to the order of $a$, like the Whitham equation does with the KdV equation. 
As a matter of fact, one may modify the argument in \cite{Lannes}, for instance, to verify that solutions of \eqref{E:BW}-\eqref{def:M} and \eqref{E:boussinesq} exist and they converge to those of the water wave problem up to the order of $a$ during a relevant interval of time. The global-in-time well-posedness for \eqref{E:boussinesq} was established in \cite{Schonbek} and \cite{Amick}, for instance. Including the {\em full dispersion} in water waves, on the other hand, \eqref{E:BW}-\eqref{def:M} may capture the breaking effects. This is the subject of investigation here. The Benjamin-Feir instability and other high frequency effects for \eqref{E:BW}-\eqref{def:M} were studied in \cite{HP2}.

If we furthermore assume that $\eta$ is much smaller than the fluid depth $=1$ then we may reject terms of the order $u\eta$ in the former equation of \eqref{E:BW} with respect to terms of the order $u$ and, after suppressing $a$, we arrive at 
\begin{equation}\label{E:main}
\begin{aligned}
&\partial_t\eta+\partial_xu+u\partial_x\eta=0,\\
&\partial_tu+\mathcal{M}\partial_x\eta+u\partial_xu=0.
\end{aligned}
\end{equation}
(Although we reject $\eta$ with respect to $1$, we must not $u\partial_x\eta$ since it is a priori not smaller than terms in the latter equation of \eqref{E:BW}.) The main result asserts the wave breaking in \eqref{E:main} and \eqref{def:M}, provided that the slope of the initial velocity is sufficiently negative. Note that 
the integral representation of $\mathcal{M}\partial_x$ may be written explicitly. Specifically,
\begin{equation} \label{def:K}
\mathcal{M}\partial_x f(x)=-\frac12{\text PV}\int_{-\infty}^\infty \frac {f(y)}{\sinh(\frac\pi2(x-y))}~dy,
\end{equation}
where $PV$ stands for the Cauchy principal value. 

\begin{theorem}[Wave breaking in \eqref{E:main}-\eqref{def:K}]\label{thm:main}
Assume that $\eta_0, u_0\in H^\infty(\mathbb{R})$. For $\epsilon>0$ sufficiently small, 
assume that 
\begin{align}
\|u_0^{(n)}\|_{L^\infty(\mathbb{R})}<&n^{(n-1)/\alpha+1}b^{n-1}, &&n=2,3,\dots,\label{A:un} \\
\|\eta_0\|_{L^\infty(\mathbb{R})}<&\frac{1}{2\epsilon}, && \label{A:h0}\\
\|\eta_0^{(n)}\|_{L^\infty(\mathbb{R})}<&\frac1\epsilon n^{n/\alpha}b^{n-1}, &&n=1,2,\dots\label{A:hn}
\end{align}
for some $b\geq 1$ and for some $\alpha$ such that $\frac12(1+\epsilon)<\alpha<\frac23(1-14\epsilon)$. Moreover, assume that 
\begin{align}
\epsilon^2(-\inf_{x\in\mathbb{R}}u_0'(x))^2>&1+\|\eta_0\|_{H^2(\mathbb{R})},\label{A:m1}\\
\epsilon\Big(\frac{1-\epsilon}{1+\epsilon}\Big)^2(-\inf_{x\in\mathbb{R}}u_0'(x))^{1/4}>&
\frac{4e}{2^{1/\alpha-1}-1}, \label{A:m2} \\
\epsilon^5(1-\epsilon)^4(-\inf_{x\in\mathbb{R}}u_0'(x))^{3/4}>&
\frac{80}{\pi}(1+(2e)^{1/\alpha}b).\label{A:m3}
\end{align}
Then the solution of \eqref{E:main}-\eqref{def:K} and 
\[
\eta(x,0)=\eta_0(x), \qquad u(x,0)=u_0(x),
\]
exhibits wave breaking. Specifically,
\[ 
|u(x,t)|<\infty \qquad \text{for all $x\in\mathbb{R}$}\quad\text{for all $t\in [0,T)$}
\]
but
\[
\inf_{x\in\mathbb{R}}\partial_xu(x,t) \to -\infty \qquad\text{as $t\to T-$}
\] 
for some $T>0$. Moreover,
\begin{equation}\label{E:T}
\frac{1}{1+\epsilon}\frac{1}{-\inf_{x\in\mathbb{R}}u_0'(x)}<T<
\frac{1}{(1-\epsilon)^2}\frac{1}{-\inf_{x\in\mathbb{R}}u_0'(x)}.
\end{equation} 
\end{theorem}

The assumptions \eqref{A:un}-\eqref{A:hn} require that $\eta_0$ and $u_0$ belong to the Gevrey class of index $1/\alpha$. Since $1/\alpha>1$, nontrivial $\eta_0$ and $u_0$ with compact support exist. They are technical assumptions and may be removed if the kernel associated with the integral representation of $\mathcal{M}$ is regular; see \cite{Seliger, CE98}, for instance. The assumptions \eqref{A:m1}-\eqref{A:m3} require that $u_0'$ be sufficiently negative somewhere in $\mathbb{R}$. The breaking scenario, we think, is that the profile of $u$ at such a point steepens until it becomes vertical in finite time.

Following along the same line as the argument in \cite{HT1, Hur-breaking} for \eqref{E:whitham}-\eqref{def:M1/2}, where $c(\xi)=|\xi|^{\alpha-1}$, $0<\alpha<1/2$, the proof of Theorem~\ref{thm:main} examines the ordinary differential equations for $u$ and its derivatives of all orders along the characteristics, which involve $\mathcal{M}\partial_x$ and $\eta$ and its derivatives of all orders along the characteristics. In other words, we examine $\eta$, $u$ and their derivatives of all orders along the characteristics (see \eqref{e:h0}-\eqref{e:vn}). To the best of the authors' knowledge, this is new. Naumkin and Shishmar\"ev \cite{NS94} made a breaking argument for related, nonlinear nonlocal equations. But it does not apply to \eqref{E:main} (or \eqref{E:BW}) because of the severe nonlinearities. 

In Lemma~\ref{lem:Kn}, we make a straightforward calculation and show that the kernel associated with \eqref{def:K} is singular of a logarithmic order near zero. To compare, the kernel associated with the integral representation of $\mathcal{M}_{1/2}\partial_x$ (see \eqref{def:M1/2}) for the Whitham equation may not be written explicitly, although it behaves like $|x|^{-1/2}$ near zero; see \cite{Hur-breaking}, for instance, and references therein. Note that $\mathcal{M}\partial_x$ is less singular than $\mathcal{M}_{1/2}\partial_x$. On the other hand, the nonlinearities of \eqref{E:main} are much more severe than that of \eqref{E:whitham}, permitting $\eta$ and its derivatives to grow large along the characteristics (see \eqref{claim:h0}-\eqref{claim:hn}), when one attempts to bound the nonlocal forcing term involving $\eta$ along the characteristics by the nonlinearity in the latter equation of \eqref{E:main}. This is why we are unable to handle the nonlinearity of \eqref{E:BW}. 
We make strong use of that the kernel associated with \eqref{def:K} less singular than a polynomial order near zero. Moreover, $\eta$ and its derivatives along the characteristics grow larger than what a logarithmic singularity can control, so that we cannot control the second derivative of $u$ along the characteristics, like in \cite{HT1} for \eqref{E:whitham}-\eqref{def:M1/2}, where $c(\xi)=|\xi|^{\alpha-1}$ and $\alpha>1/3$. We exploit the ``smoothing effects" of the characteristics when the derivative of $u$ is sufficiently negative (see \eqref{claim:X2} and \eqref{claim:X3}). 

It is physically more satisfying to prove wave breaking for $\eta$, rather than $u$. We believe that $\eta$ breaks when $u$ does. The proof of Theorem~\ref{thm:main}, however, does not explore blowup in the former equation of \eqref{E:main}. Moreover, it is desirable to prove wave breaking in \eqref{E:BW}, rather than \eqref{E:main}. This is a subject of future investigation.

\subsection*{Remarks on other Boussinesq-Whitham models}
Perhaps, the best known among Boussinesq's equations in the shallow water theory is
\begin{equation}\label{E:Bnesq1}
\partial_t^2\eta=\partial_x^2\eta+\frac13a\partial_x^4\eta+\frac32a\partial_x^2(\eta^2).
\end{equation}
Including the full dispersion in water waves, one may follow Whitham's heuristics and replace the square of the phase speed~$=1-\frac13a\xi^2$ by that of water waves (see \eqref{def:cWW}). The result becomes
\begin{equation}\label{E:BW1}
\partial_t^2\eta=\mathcal{M}\partial_x^2\eta+\frac32a\partial_x^2(\eta^2),
\end{equation}
where $\mathcal{M}$ is in \eqref{def:M}. It is one of many which stake the claim to the ``Boussinesq-Whitham equation." 
Unfortunately, the initial value problem associated with the linear part of \eqref{E:BW1} is ill-posed in the periodic setting. Hence, it is not suitable for the purpose of describing wave packet propagation.

Under the assumption $\partial_t\eta+\partial_x\eta=O(a)$, \eqref{E:Bnesq1} is formally equivalent to
\[
\partial_t^2\eta=\frac13\partial_t^2\partial_x^2\eta+\partial_x^2\eta+\frac32a\partial_x^2(\eta^2)
\]
up to the order of $a$. 
Including the full dispersion in water waves, likewise, one arrives at
\begin{equation}\label{E:BW2}
\partial_t^2\eta=\mathcal{M}\Big(\partial_x^2\eta+\frac32a\partial_x^2(\eta^2)\Big).
\end{equation} 
The initial value problem for \eqref{E:BW2} is well-posed at least locally in time. But it fails to explain the Benjamin-Feir instability; 
see \cite{HP1}, for instance, for details. Hence, it is a poor candidate for the purpose of studying the stability of Stokes waves. In contrast, one of the authors~\cite{HP2} proved the Benjamin-Feir instability in \eqref{E:BW}-\eqref{def:M}. 

Saut~\cite{Saut} (see also \cite{Dobrokhotov}) alternatively proposed
\begin{equation}\label{E:BW3}
\begin{aligned}
&\partial_t\eta+\mathcal{M}\partial_xu+a\partial_x(u\eta)= 0,\\
&\partial_tu+\partial_x\eta+a\,u\partial_xu=0,
\end{aligned}
\end{equation}
as Boussinesq-Whitham equations. They are formally equivalent to \eqref{E:BW}-\eqref{def:M} up to the order of $a$. But, to the best of the authors' knowledge, the well-posedness issue for \eqref{E:BW3} has not been studied. In contrast, in Section~\ref{sec:existence}, we establish the local-in-time well-posedness for \eqref{E:main}-\eqref{def:K}.

To conclude, \eqref{E:BW} (or \eqref{E:main}) is preferred over other Boussinesq-Whitham models for the purpose of studying the breaking and stability of water waves.

\section{Local well-posedness}\label{sec:existence}

We discuss the initial value problem associated with \eqref{E:main}-\eqref{def:K} or, equivalently,
\begin{equation}\label{E:main'}
\begin{aligned}
&\partial_t\eta+\partial_xu+u\partial_x\eta=0,\\
&\partial_tu-\mathcal{H}h+\mathcal{R}\eta+u\partial_xu=0.
\end{aligned}
\end{equation}
Here $\mathcal{H}$ denotes the Hilbert transform, defined as a Fourier multiplier as 
\[
\widehat{\mathcal{H}f}(\xi)=-i\text{sgn}(\xi)\widehat{f}(\xi).
\] 
Since
\[
|\text{sgn}(\xi)-\tanh(\xi)|\leq e^{-|\xi|} \qquad\text{pointwise in $\mathbb{R}$}
\]
by a direct calculation (see \cite[Lemma~2.15]{Yosihara}, for instance), we find that 
\begin{equation}\label{I:R}
\|\mathcal{R}f\|_{H^s(\mathbb{R})}\leq C\|f\|_{L^2(\mathbb{R})} \qquad\text{for any $s\geq 0$},
\end{equation}
where $C>0$ a constant is independent of $f$. 


\begin{theorem}[Local well-posedness]\label{T:exist}
If $\eta_0\in H^s(\mathbb{R})$ and $u_0\in H^{s+1/2}(\mathbb{R})$ for $s> 2$ then a unique solution of \eqref{E:main}-\eqref{def:K},
\[
\eta(x,0)=\eta_0(x)\quad\text{and}\quad u(x,0)=u_0(x),
\]
exists in $H^s(\mathbb{R})\times H^{s+1/2}(\mathbb{R})$ during the interval of time $[0,T)$ for some $T>0$. Moreover, $(\eta_0,u_0)\mapsto (\eta(t),u(t))$ is continuous on $H^s(\mathbb{R})\times H^{s+1/2}(\mathbb{R})$ for all $t\in [0,T)$.
\end{theorem}

Combining an a priori bound and a compactness argument, one may be able to establish local-in-time well-posedness for \eqref{E:shallow} in $H^s(\mathbb{R})\times H^{s+1/2}(\mathbb{R})$, $s>2$; see \cite{Kato}, for instance, for details. Without recourse to the dispersion effects, the argument in \cite{Kato} works for \eqref{E:main'}-\eqref{I:R} mutatis mutandis. Below we merely include how one obtains a priori bound for \eqref{E:main'}-\eqref{I:R}, and we omit other parts of the proof.

\

Note that $\|\mathcal{H}f\|_{L^2(\mathbb{R})}=\|f\|_{L^2(\mathbb{R})}$ and $\mathcal{H}^2=-1$. Note that $\Lambda:=\mathcal{H}\partial_x$ is self-adjoint and linked with half-integer Sobolev spaces. Specifically,
\[
\int^\infty_{-\infty}(f^2+f\Lambda f)~dx
\]
is equivalent to $\|f\|_{H^{1/2}(\mathbb{R})}^2$. Moreover the commutator of $\Lambda$ is ``smoothing."

\begin{lemma}\label{L:smoothing}
It follows that
\begin{equation}\label{E:smoothing}
\int^\infty_{-\infty}af\mathcal{H}\partial_x f~dx\leq C\|a\|_{H^{3/2+}(\mathbb{R})}\|f\|_{H^{1/2}(\mathbb{R})}^2
\quad\text{and}\quad
\int^\infty_{-\infty}a(\partial_xf)\mathcal{H}\partial_xf~dx\leq C\|a\|_{H^{5/2+}(\mathbb{R})}\|f\|_{L^2(\mathbb{R})}^2,
\end{equation}
where $C>0$ a constant is independent of $f$ and $a$.
\end{lemma}

\begin{proof}
Note that $\Lambda^{1/2}$ is self-adjoint, and we calculate that 
\[
\int af\mathcal{H}\partial_xf~dx=\int a(\Lambda^{1/2}f)^2~dx+\int(\Lambda^{1/2}[\Lambda^{1/2},a]f)f~dx.
\]
Clearly, the first term of the right side is bounded by $\|a\|_{L^\infty}\|\Lambda^{1/2}f\|_{L^2}^2$. We claim that the second term of the right side is bounded by $\||\xi|\widehat{a}\|_{L^1}\|f\|_{H^{1/2}}^2$ up to multiplication by a constant. Indeed, since 
\[
(\Lambda^{1/2}[\Lambda^{1/2},a]f)^\wedge(\xi)
=\frac{1}{\sqrt{2\pi}}\int^\infty_{-\infty}|\xi|^{1/2}(|\xi|^{1/2}-|\xi_1|^{1/2})\widehat{a}(\xi-\xi_1)\widehat{f}(\xi_1)~d\xi_1
\]
and since $|\xi|^{1/2}||\xi|^{1/2}-|\xi_1|^{1/2}|\leq C|\xi-\xi_1|$ for all $\xi,\xi_1\in\mathbb{R}$ for some constant $C>0$ by a direct calculation (see the proof of \cite[Lemma~2.14]{Yosihara}, for instance), Young's inequality and the Parseval theorem assert that 
\[
\|\Lambda^{1/2}[\Lambda^{1/2},a]f\|_{L^2}\leq C\||\xi|\widehat{a}\|_{L^1}\|f\|_{L^2}
\]
for some constant $C>0$ independent of $f$ and $a$. H\"older's inequality therefore proves the claim. The first inequality of \eqref{E:smoothing} then follows by the Sobolev inequality. 

Note that $\mathcal{H}$ is skew-adjoint, and we calculate that 
\begin{align*}
\int a(\partial_xf)\mathcal{H}\partial_xf~dx
=&-\int a(\mathcal{H}\partial_xf)\partial_xf~dx-\int([\mathcal{H},a]\partial_xf)\partial_xf~dx\\
=&-\frac12\int([\mathcal{H},a]\partial_xf)\partial_xf~dx.
\end{align*}
Since 
\[
(\partial_x[\mathcal{H},a]\partial_xf)^\wedge(\xi)
=-\frac{1}{\sqrt{2\pi}}\int^\infty_{-\infty}\xi(\text{sgn}(\xi)-\text{sgn}(\xi_1))\widehat{a}(\xi-\xi_1)\xi_1\widehat{f}(\xi_1)~d\xi_1
\]
and since $|\xi|+|\xi_1|\leq|\xi-\xi_1|$ when $\text{sgn}(\xi)\neq\text{sgn}(\xi_1)$ by a direct calculation (see the proof of \cite[Lemma~2.14]{Yosihara}, for instance), Young's inequality and the Parseval theorem assert that
\[
\|\partial_x[\mathcal{H},a]\partial_xf\|_{L^2}\leq \frac12\||\xi|^2\widehat{a}\|_{L^1}\|f\|_{L^2}.
\]
H\"older's inequality  and the Sobolev inequality then prove the second inequality of \eqref{E:smoothing}. This completes the proof.
\end{proof}

To proceed, for $k\geq 1$ an integer, let
\begin{equation}\label{def:energy}
E_k^2(t)=\frac12\|\eta(t)\|_{L^2(\mathbb{R})}^2+\frac12\|u(t)\|_{L^2(\mathbb{R})}^2+\sum_{j=1}^ke_j(t),
\end{equation}
where 
\begin{equation}\label{def:Ej}
e_j(t)=\frac12\int^\infty_{-\infty}((\partial_x^j\eta(t))^2+(\partial_x^ju(t))\Lambda(\partial_x^ju(t)))~dx.
\end{equation}
Note tha $E_k(t)$ is equivalent to $\|\eta(t)\|_{H^k(\mathbb{R})}+\|u(t)\|_{H^{k+1/2}(\mathbb{R})}$.

\begin{lemma}[A priori bound]\label{lem:energy}
If $\eta\in H^k(\mathbb{R})$ and $u\in H^{k+1/2}(\mathbb{R})$, for $k\geq 2$ an integer, solve \eqref{E:main'}-\eqref{I:R} during the interval of time $[0,T)$ for some $T>0$ then
\begin{equation}\label{I:energy}
E_k(t)\leq \frac{E_k(0)}{1-CE_k(0)t}
\end{equation}
for all $t\in[0,T']$, $0<T'<T$, where $C>0$ a constant is independent of $\eta$ and $u$, and $T'$ depends upon $E_k(0)$. Moreover,
\begin{equation}\label{I:norm}
\|\eta(t)\|_{H^k(\mathbb{R})}+\|u(t)\|_{H^{k+1/2}(\mathbb{R})}\leq C(t, \|\eta(0)\|_{H^k(\mathbb{R})}, \|u(0)\|_{H^{k+1/2}(\mathbb{R})})
\end{equation}
for all $t\in[0,T']$.
\end{lemma}

\begin{proof}
For $j\geq 1$ an integer, differentiating \eqref{def:Ej} in time and using \eqref{E:main'}, we arrive at
\begin{align*}
\frac{de_j}{dt}=&\int((\partial_t\partial_x^j\eta)(\partial_x^j\eta)
+(\partial_t\partial_x^ju)\Lambda(\partial_x^ju))~dx \notag \\
=&-\int\partial_x^j(\partial_xu+u\partial_x\eta)(\partial_x^j\eta)~dx
-\int\partial_x^j(-\mathcal{H}\eta+\mathcal{R}\eta+u\partial_xu)\Lambda(\partial_x^ju)~dx \notag \\
=:&(I)+(II)
\end{align*}
during the interval of time $(0,T)$. An integration by parts leads to that 
\begin{align}\label{e:1}
(I)=-\int(\partial_x^{j+1}u)(\partial_x^j\eta)~dx&+\frac12\int(\partial_xu)(\partial_x^j\eta)^2~dx \\
&-\int(\partial_x^j(u\partial_x\eta)-u(\partial_x^{j+1}\eta))(\partial_x^j\eta)~dx. \notag
\end{align}
Since $\Lambda=\mathcal{H}\partial_x$, $\mathcal{H}$ is skew-adjoint and $\mathcal{H}^2=-1$, moreover,
\begin{align}\label{e:2}
(II)=&-\int(\partial_x^{j+1}\eta)(\partial_x^ju)~dx-\int(\mathcal{H}\partial_x^{j+1}\mathcal{R}\eta)(\partial_x^ju)~dx\\
&-\int u(\partial_x^{j+1}u)\mathcal{H}(\partial_x^{j+1}u)~dx
-j\int(\partial_xu)(\partial_x^ju)(\mathcal{H}\partial_x^{j+1}u)~dx \notag \\
&-\int(\partial_x^j(u\partial_xu)-u(\partial_x^{j+1}u)-j(\partial_xu)(\partial_x^ju))\Lambda(\partial_x^ju)~dx. \notag
\end{align}

Note that the first term of the right side of \eqref{e:1} and the first term of the right side of \eqref{e:2} cancel each other when added together after an integration by parts. Note that the second term of the right side of \eqref{e:1} is bounded by $\frac12\|\partial_xu\|_{L^\infty}\|\partial_x^j\eta\|_{L^2}^2$, and the last term of the right side of \eqref{e:1} is bounded by $\|u\|_{H^j}\|\partial_x^j\eta\|_{L^2}^2$ up to multiplication by a constant by the Leibniz rule. Note that the second term of the right side of \eqref{e:2} is bounded by $\|\eta\|_{L^2}\|\partial_x^ju\|_{L^2}$ by \eqref{I:R}, and the third and the fourth terms of the right side of \eqref{e:2} are bounded by $\|u\|_{H^{5/2+}}\|\partial_x^ju\|_{H^{1/2}}^2$ by \eqref{E:smoothing}. Moreover, note that the last term of the right side of \eqref{e:2}, for $j\geq 2$ an integer, is bounded by $\|u\|_{H^{j+1/2}}^2\|\Lambda^{1/2}\partial_x^ju\|_{L^2}$ up to multiplication by a constant by the fractional Leibniz rule and the Sobolev inequality. To recapitulate, 
\begin{equation}\label{E:ee}
\frac{de_j}{dt}\leq C(1+\|u\|_{H^{5/2+}}+\|u\|_{H^{j+1/2}})(\|\eta\|_{H^j}^2+\|u\|_{H^{j+1/2}}^2)
\end{equation}
for $j\geq 2$ an integer during the interval of time $(0,T)$, where $C>0$ a constant is independent of $\eta$ and $u$. 

To proceed, we use \eqref{E:main'} and integrate by parts to show that
\begin{align}
\frac12\frac{d}{dt}\|\eta\|_{L^2}^2=&-\int(\partial_xu+u\partial_x\eta)\eta~dx\leq\|\partial_xu\|_{L^2}\|\eta\|_{L^2}
+\frac12\|\partial_xu\|_{L^\infty}\|\eta\|_{L^2}^2, \label{E:eh} \\
\frac12\frac{d}{dt}\|u\|_{L^2}^2=&\int(\mathcal{H}\eta-\mathcal{R}\eta-u\partial_xu)u~dx\leq2\|\eta\|_{L^2}\|u\|_{L^2}+\|\partial_xu\|_{L^\infty}\|u\|_{L^2}^2 \label{E:eu}
\end{align}
during the interval of time $(0,T)$. Adding \eqref{E:ee} through \eqref{E:eu}, we deduce that 
\[
\frac{dE_k}{dt}\leq C E_k^2
\]
for $k\geq 2$ an integer during the interval of time $(0,T)$, where $C>0$ a constant is independent of $\eta$ and $u$. Therefore \eqref{I:energy} follows because it invites a solution until the time $T'=(CE_k(0))^{-1}$. Furthermore \eqref{I:norm} follows because $E_k(t)$ is equivalent to $\|\eta(t)\|_{H^k}+\|u(t)\|_{H^{k+1/2}}$. This completes the proof.
\end{proof}

\section{Proof of Theorem~\ref{thm:main}}\label{sec:proof}
We assume that $\eta_0$ and $u_0$ satisfy \eqref{A:un}-\eqref{A:hn}, \eqref{A:m1}-\eqref{A:m3}. Let $\eta$ and $u$ be the unique solution of \eqref{E:main}-\eqref{def:K},  
\[
\eta(x,0)=\eta_0(x)\quad\text{and}\quad u(x,0)=u_0(x),
\] 
in $C^\infty([0,T);H^\infty(\mathbb{R})\times H^\infty(\mathbb{R}))$ for some $T>0$. We assume that $T$ is the maximal time of existence.

\

For $x\in\mathbb{R}$, let $X(t;x)$ solve
\begin{equation}\label{def:X}
\frac{dX}{dt}(t;x)=u(X(t;x),t)\quad\text{and}\quad X(0;x)=x.
\end{equation}
Since $u(x,t)$ is bounded and satisfies a Lipschitz condition in $x$ for all $x\in\mathbb{R}$ for all $t\in[0,T)$, it follows from the ODE theory that $X(\cdot\,;x)$ is continuously differentiable throughout the interval $(0,T)$ for all $x\in\mathbb{R}$. Since $u(x,t)$ is smooth in $x$ for all $x\in\mathbb{R}$ for all $t\in[0,T)$, furthermore, $x\mapsto X(\cdot\,;x)$ is infinitely continuously differentiable throughout the interval $(0,T)$ for all $x\in\mathbb{R}$. 

Let
\begin{equation}\label{def:vh_n}
\zeta_n(t;x)=(\partial_x^n\eta)(X(t;x),t)\quad\text{and}\quad v_n(t;x)=(\partial_x^nu)(X(t;x),t)
\end{equation}
for $n=0,1,2,\dots$. Differentiating \eqref{E:main} with respect to $x$ and evaluating the result at $x=X(t;x)$, we arrive at
\begin{align}
&\frac{d\zeta_0}{dt}+v_1=0,& & \label{e:h0} \\
&\frac{d\zeta_n}{dt}+\sum_{j=1}^n\left(\begin{matrix}n\\j\end{matrix}\right)v_j\zeta_{n+1-j}+v_{n+1}=0
& &\text{for $n=1,2,\dots$}, \label{e:hn}
\intertext{and}
&\frac{dv_0}{dt}+K_0(t;x)=0, & & \label{e:v0} \\
&\frac{dv_n}{dt}+\sum_{j=1}^n\left(\begin{matrix}n\\j\end{matrix}\right)v_jv_{n+1-j}+K_n(t;x)=0 
& &\text{for $n=1,2,\dots$}. \label{e:vn}
\end{align}
Here and elsewHere $\left(\begin{matrix}n\\j\end{matrix}\right)$'s are the binomial coefficients, and 
\begin{align*}
K_n(t;x)=&(\mathcal{M}\partial_x^{n+1}\eta)(X(t;x),t) \notag\\
=&-\frac12\int^\infty_{-\infty}\text{csch}(\tfrac\pi2(X(t;x)-y))((\partial_x^n\eta)(X(t;x),t)-(\partial_x^n\eta)(y,t))~dy
\end{align*}
for $n=0,1,2 \dots$ (see \eqref{def:K}). Since $u(x,t)$ is smooth, square integrable in $x$, and smooth in $t$ for all $x\in\mathbb{R}$  for all $t\in[0,T)$, and since $X(t;x)$ is continuously differentiable in $t$ and smooth in $x$ for all $t\in[0,T)$ for all $x\in\mathbb{R}$, it follows that $K_n(t;x)$ is continuously differentiable in $t$ and smooth in $x$ for all $t\in[0,T)$ for all $x\in\mathbb{R}$. 

\begin{lemma}\label{lem:Kn}
Let $0<\delta<1$. For $\epsilon>0$ is sufficiently small,
\begin{equation}\label{e:Kn}
|K_n(t;x)|<\frac{40}{\pi}\frac1\epsilon
(\delta^{-\epsilon}\|\zeta_n(t)\|_{L^\infty(\mathbb{R})}+\delta^{1-\epsilon}\|\zeta_{n+1}(t)\|_{L^\infty(\mathbb{R})}),\qquad n=0,1,2,\dots
\end{equation}
for all $t\in[0,T)$ for all $x\in\mathbb{R}$.
\end{lemma}

The proof involves direct calculations of \eqref{def:K}. We include the detail in Appendix~\ref{sec:appendix}.

\

Let
\begin{equation}\label{def:q}
m(t)=\inf_{x\in\mathbb{R}}v_1(t;x)=\inf_{x\in\mathbb{R}}(\partial_xu)(x,t)=:m(0)q^{-1}(t).
\end{equation}
Note that $v_1(t;\,\cdot)$ and, hence, $m(t)$ are continuous for all $t\in[0,T)$. 
Clearly, $m(t)<0$ for all $t\in[0,T)$, $q(0)=1$ and $q(t)>0$ for all $t\in[0,T)$. Indeed, $m(t)\geq0$ would imply that $u(\cdot\,,t)$ be non-decreasing in $\mathbb{R}$ and, hence, $u(\cdot\,,t)\equiv0$. 

\

We shall show that 
\begin{equation}\label{claim:A}
|K_1(t;x)|<\epsilon^2m^2(t)\qquad\text{for all $t\in[0,T)$ for all $x\in\mathbb{R}$}.
\end{equation}
Since\footnote{Note in passing that $-\mathcal{M}\partial_x$ is the Hilbert transform for the infinite horizontal strip of unit depth, subject to the Neumann boundary condition at the bottom.} 
$\|\mathcal{M}\partial_xf\|_{L^2(\mathbb{R})}\leq \|f\|_{L^2(\mathbb{R})}$ by the Parseval theorem, it follows from \eqref{A:m1} and the Sobolev inequality that 
\[
|K_1(0;x)|=|\mathcal{M}\eta_0''(x)|\leq \|\eta_0\|_{H^{3/2+}(\mathbb{R})}<\epsilon^2m^2(0)\qquad\text{for all $x\in\mathbb{R}$}.
\]
That is, \eqref{claim:A} holds at $t=0$. 
Suppose on the contrary that $|K_1(T_1;x)|=\epsilon^2m^2(T_1)$ for some $T_1\in(0,T)$ for some $x\in\mathbb{R}$. By continuity, we may assume, without loss of generality, that 
\begin{equation}\label{I:A}
|K_1(t;x)|\leq \epsilon^2m^2(t)\qquad\text{for all $t\in[0,T)$ for all $x\in\mathbb{R}$}.
\end{equation}
We seek a contradiction.

\

Below we gather some preliminaries.

\begin{lemma}\label{lem:S}
Let $0<\gamma<1$. For $t\in[0,T_1]$, let
\begin{equation}\label{def:Sigma}
\Sigma_\gamma(t)=\{x\in\mathbb{R}: v_1(t;x)\leq (1-\gamma)m(t)\}.
\end{equation}
If $0<\epsilon\leq\gamma<1/2$ for $\epsilon>0$ sufficiently small then $\Sigma_\gamma(t_2)\subset\Sigma_\gamma(t_1)$ whenever $0\leq t_1\leq t_2\leq T_1$.
\end{lemma}

The proof is very similar to that of \cite[Lemma~2.1]{HT1}. We include the detail in Appendix~\ref{sec:appendix} for completeness.

\begin{lemma}\label{lem:q}
$0<q(t)\leq 1$ and it is decreasing for all $t\in[0,T_1]$.
\end{lemma}

\begin{proof}
The proof is very similar to that of \cite[Lemma~2.2]{HT1}. Here we include the detail for future references.

For $0<\epsilon\leq\gamma<1/2$, $\epsilon>0$ sufficiently small, let $x\in\Sigma_\gamma(T_1)$, and we suppress it to simplify the exposition. Note from \eqref{def:q} and Lemma~\ref{lem:S} that 
\begin{equation}\label{I:mv}
m(t)\leq v_1(t)\leq (1-\gamma)m(t)<0\qquad \text{for all $t\in[0,T_1]$}.
\end{equation}
One may write the solution of \eqref{e:vn}, where $n=1$, as
\begin{equation}\label{def:r}
v_1(t)=\frac{v_1(0)}{1+v_1(0)\int^t_0(1+(v_1^{-2}K_1)(\tau))~d\tau}=:m(0)r^{-1}(t).
\end{equation}
Clearly, $r(t)$ is continuously differentiable and $r(t)>0$ for all $t\in[0,T_1]$. Since 
\[
|(v_1^{-2}K_1)(t)|<(1-\gamma)^{-2}\epsilon^2<\epsilon\qquad\text{for all $t\in[0,T_1]$}
\]
for $\epsilon>0$ sufficiently small, by \eqref{I:mv} and \eqref{I:A}, we infer from \eqref{def:r} that 
\begin{equation}\label{I:dr/dt}
(1+\epsilon)m(0)\leq \frac{dr}{dt}\leq (1-\epsilon)m(0)<0
\end{equation}
throughout the interval $(0,T_1)$. Consequently, $r(t)$ and, hence, $v_1(t)$ (see \eqref{def:r}) are decreasing for all $t\in[0,T_1]$. Furthermore, $m(t)$ and, hence, $q(t)$ (see \eqref{def:q}) are decreasing for all $t\in[0,T_1]$. This completes the proof. By the way, note from \eqref{def:q}, \eqref{def:r} and \eqref{I:mv} that
\begin{equation}\label{I:qr}
q(t)\leq r(t)\leq \frac{1}{1-\gamma}q(t)\qquad\text{for all $t\in[0,T_1]$}.
\end{equation} 
\end{proof}

\begin{lemma}\label{lem:qs}
For $s>0$, $s\neq1$, and for $t\in[0,T_1]$, 
\begin{align}
\int^t_0q^{-s}(\tau)~d\tau\leq &-\frac{1}{s-1}\frac{1}{(1-\epsilon)^{1+s}}\frac{1}{m(0)}
\Big(q^{1-s}(t)-\frac{1}{(1-\epsilon)^{1-s}}\Big). \label{I:s>1}
\intertext{For $t\in[0,T_1]$, }
\int^t_0q^{-1}(\tau)~d\tau\leq &-\frac{1}{(1-\epsilon)^2}\frac{1}{m(0)}\Big(\log\frac{1}{1-\epsilon}-\log q(t)\Big).
\label{I:s=1}
\end{align}
\end{lemma}

The proof is found in \cite[Lemma~2.3]{HT1}, for instance; see also the proof of \eqref{I:s>1h} below. 

\

To proceed, we shall show that 
\begin{align}
\|v_0(t)\|_{L^\infty(\mathbb{R})}=&\|u(t)\|_{L^\infty(\mathbb{R})}<C_0,\label{claim:v0} \\
\|v_1(t)\|_{L^\infty(\mathbb{R})}=&\|(\partial_xu)(t)\|_{L^\infty(\mathbb{R})}<C_1q^{-1}(t),\label{claim:v1}\\
\|v_n(t)\|_{L^\infty(\mathbb{R})}=&\|(\partial_x^nu)(t)\|_{L^\infty(\mathbb{R})}
<C_2n^{(n-1)/\alpha+1}b^{n-1}q^{-1-(n-1)\sigma}(t)\label{claim:vn} 
\end{align}
for $n=2,3,\dots$, and 
\begin{align}
\|\zeta_0(t)\|_{L^\infty(\mathbb{R})}=&\|\eta(t)\|_{L^\infty(\mathbb{R})}
<\frac{C_2}{\epsilon}q^{-\epsilon}(t),\label{claim:h0}\\
\|\zeta_n(t)\|_{L^\infty(\mathbb{R})}=&\|(\partial_x^n\eta)(t)\|_{L^\infty(\mathbb{R})}
<\frac{C_2}{\epsilon}n^{n/\alpha}b^{n-1}q^{-\epsilon-n\sigma}(t)\label{claim:hn}
\end{align}
for $n=1,2,\dots$ for all $t\in[0,T_1]$. Here 
\begin{equation}\label{def:C}
C_0=2(\|u_0\|_{L^\infty(\mathbb{R})}+\|u_0'\|_{L^\infty(\mathbb{R})}),\qquad 
C_1=2\|u_0'\|_{L^\infty(\mathbb{R})},\qquad
C_2=(-m(0))^{3/4},
\end{equation}
and
\begin{equation}\label{I:sigma}
\frac12(1+\epsilon)<\alpha< \frac23(1-14\epsilon)\quad\text{and}\quad\sigma=\frac32+6\epsilon
\end{equation}
so that 
\[
\sigma\alpha<1-10\epsilon.
\] 
Throughout the proof, we use
\[
C_0>C_1\quad\text{and}\quad \frac12C_1=\|u_0'\|_{L^\infty(\mathbb{R})}>C_2>1
\]
to simplify the exposition. It follows from \eqref{def:C}, \eqref{def:q}, \eqref{A:un} and \eqref{A:h0}, \eqref{A:hn} that
\begin{align*}
\|v_0(0)\|_{L^\infty(\mathbb{R})}=&\|u_0\|_{L^\infty(\mathbb{R})}<C_0, \\
\|v_1(0)\|_{L^\infty(\mathbb{R})}=&\|u_0'\|_{L^\infty(\mathbb{R})}<C_1=C_1q^{-1}(0),\\
\|v_n(0)\|_{L^\infty(\mathbb{R})}=&\|u_0^{(n)}\|_{L^\infty(\mathbb{R})}
<C_2n^{(n-1)/\alpha+1}b^{n-1}q^{-1-(n-1)\sigma}(0)
\intertext{for $n=2,3,\dots$, and } 
\|\zeta_0(0)\|_{L^\infty(\mathbb{R})}=&\|\eta_0\|_{L^\infty(\mathbb{R})}
<\frac{1}{2\epsilon}<\frac{C_2}{\epsilon}q^{-\epsilon}(0), \\
\|\zeta_n(0)\|_{L^\infty(\mathbb{R})}=&\|\eta^{(n)}\|_{L^\infty(\mathbb{R})}
<\frac{C_2}{\epsilon}n^{n/\alpha}b^{n-1}q^{-\epsilon-n\sigma}(0)
\end{align*}
for $n=1,2,\dots$. That is, \eqref{claim:v0}-\eqref{claim:vn} and \eqref{claim:h0}-\eqref{claim:hn} hold for all $n=0,1,2,\dots$ at $t=0$. Suppose on the contrary that \eqref{claim:v0}-\eqref{claim:vn} and \eqref{claim:h0}-\eqref{claim:hn} hold for all $n=0,1,2,\dots$ throughout the interval $[0,T_2)$, but one of the inequalities fails for some $n$ at $t=T_2$ for some $T_2\in(0,T_1]$. By continuity, we may assume that 
\begin{align}
\|v_0(t)\|_{L^\infty(\mathbb{R})}\leq&C_0,\label{I:v0} \\
\|v_1(t)\|_{L^\infty(\mathbb{R})}\leq&C_1q^{-1}(t),\label{I:v1}\\
\|v_n(t)\|_{L^\infty(\mathbb{R})}\leq&C_2n^{(n-1)/\alpha+1}b^{n-1}q^{-1-(n-1)\sigma}(t)\label{I:vn} 
\intertext{for $n=2,3,\dots$, and}
\|\zeta_0(t)\|_{L^\infty(\mathbb{R})}\leq&\frac{C_2}{\epsilon}q^{-\epsilon}(t),\label{I:h0} \\
\|\zeta_n(t)\|_{L^\infty(\mathbb{R})}\leq&\frac{C_2}{\epsilon}n^{n/\alpha}b^{n-1}q^{-\epsilon-n\sigma}(t)\label{I:hn}
\end{align}
for $n=1,2,\dots$ for all $t\in[0,T_2]$. We seek a contradiction.

\subsection*{Proof of \eqref{claim:h0}}
We integrate \eqref{e:h0} over the interval $[0,T_2]$ to show that 
\begin{align*}
|\zeta_0(T_2;x)|\leq&\|\eta_0\|_{L^\infty}+\int^{T_2}_0|v_1(t;x)|~dt\\
<&\frac{1}{2\epsilon}+C_1\int^{T_2}_0q^{-1}(t)~dt \\
\leq&\frac{C_2}{2\epsilon}-C_1\frac{1}{(1-\epsilon)^2}\frac{1}{m(0)}
\Big(\log\frac{1}{1-\epsilon}-\log q(T_2)\Big)\\
<&\frac{C_2}{2\epsilon}+2\Big(\frac{1+\epsilon}{1-\epsilon}\Big)^2\frac1\epsilon q^{-\epsilon}(T_2) \\
\leq&\frac{C_2}{2\epsilon}q^{-\epsilon}(T_2)
+\frac2\epsilon\Big(\frac{1+\epsilon}{1-\epsilon}\Big)^2q^{-\epsilon}(T_2)\\
<&\frac{C_2}{\epsilon}q^{-\epsilon}(T_2)
\end{align*}
for all $x\in\mathbb{R}$. Therefore \eqref{claim:h0} holds throughout the interval $[0,T_2]$. Here the second inequality uses \eqref{A:h0} and \eqref{I:v1}, the third inequality uses that $C_2>1$ and \eqref{I:s=1}, the fourth inequality uses that 
\[
\log\frac{1}{1-\epsilon}<2\epsilon\quad\text{and}\quad-\log x<\frac1\epsilon x^{-\epsilon}
\] 
throughout $0<x<1$ for all $0<\epsilon<1$, by direct calculations, and Lemma~\ref{lem:q}. Moreover, we assume, without loss of generality, that $\|u_0'\|_{L^\infty}=-m(0)$. The fifth inequality uses Lemma~\ref{lem:q}, and the last inequality uses \eqref{A:m2}. Indeed, 
\[
\Big(\frac{1-\epsilon}{1+\epsilon}\Big)^2(-m(0))^{3/4}>4.
\]

\subsection*{Proof of \eqref{claim:hn}}
We gather some more preliminaries. 

For $n\geq 1$, let
\begin{equation}\label{def:T3}
v_1(T_3;x)=m(T_3)\quad\text{and}\quad 
v_1(t;x)\leq \frac{1}{(1+\epsilon)^{1/(1+\epsilon+n\sigma)}}m(t)
\end{equation}
for all $t\in[T_3,T_2]$ for some $T_3\in (0,T_2)$ and for some $x\in\mathbb{R}$. Indeed, since $v_1$ and $m$ are uniformly continuous throughout the interval $[0,T_2]$, we may choose $T_3$ sufficiently close to $T_2$ so that the latter inequality of \eqref{def:T3} holds for all $t\in[T_3,T_2]$ for $\epsilon>0$ sufficiently small. We repeat the argument in the proof of Lemma~\ref{lem:q} to find that 
\begin{equation}\label{I:dr/dt'}
(1+\epsilon)m(0)\leq \frac{dr}{dt}\leq (1-\epsilon)m(0)
\end{equation}
throughout the interval $(T_3,T_2)$ for some $\epsilon>0$ sufficiently small and
\begin{equation}\label{I:qr'}
q(t)\leq r(t)\leq (1+\epsilon)^{1/(1+\epsilon+n\sigma)}q(t)
\end{equation}
for all $t\in[T_3,T_2]$. Moreover we calculate that 
\begin{align}
\int^{T_2}_{T_3}q^{-1-\epsilon-n\sigma}(t)~dt
&\leq (1+\epsilon)\int^{T_2}_{T_3} r^{-1-\epsilon-n\sigma}(t)~dt\notag\\
&\leq \frac{1+\epsilon}{1-\epsilon}\frac{1}{m(0)}
\int^{T_2}_{T_3} r^{-1-\epsilon-n\sigma}(t)\frac{dr}{dt}(t)~dt \notag \\
&=-\frac{1}{\epsilon+n\sigma}\frac{1+\epsilon}{1-\epsilon}\frac{1}{m(0)}
(r^{-\epsilon-n\sigma}(T_2)-r^{-\epsilon-n\sigma}(T_3))\notag \\
&\leq-\frac{1}{\epsilon+n\sigma}\frac{1+\epsilon}{1-\epsilon}\frac{1}{m(0)}
(q^{-\epsilon-n\sigma}(T_2)-q^{-\epsilon-n\sigma}(T_3)). \label{I:s>1h}
\end{align}
It offers a refinement over \eqref{I:s>1} when $T_3$ and $T_2$ are sufficiently close. Here the first inequality uses \eqref{I:qr'}, the second inequality uses \eqref{I:dr/dt'}, and the last inequality uses \eqref{I:qr'} and \eqref{def:T3}. 

\subsection*{Proof of \eqref{claim:hn} for $n=1$}
Let $|\zeta_1(T_2;x_1)|=\max_{x\in\mathbb{R}}|\zeta_1(T_2;x)|$. We may assume, without loss of generality, that $\zeta_1(T_2;x_1)>0$. Since $\zeta_1$ is uniformly continuous throughout the interval $[0,T_2]$, we may choose $T_3$ close to $T_2$ so that 
\begin{equation}\label{I:h1>0}
\zeta_1(t;x_1)\geq0\qquad \text{for all $t\in[T_3,T_2]$}.
\end{equation}
Moreover, we may choose $T_3$ closer to $T_2$, if necessary, so that \eqref{def:T3} and \eqref{I:s>1h} hold throughout the interval $[T_3,T_2]$. Note from \eqref{e:hn} that
\begin{align*}
\frac{d\zeta_1}{dt}(t;x_1)=&-v_1(t;x_1)\zeta_1(t;x_1)-v_2(t;x_1) \\
\leq&-m(0)\frac{C_2}{\epsilon}q^{-1}(t)q^{-\epsilon-\sigma}(t)+C_22^{1/\alpha+1}bq^{-1-\sigma}(t) \\
\leq&(-m(0)+2^{1/\alpha+1}b\epsilon)\frac{C_2}{\epsilon}q^{-1-\epsilon-\sigma}(t)
\end{align*}
for all $t\in(T_3,T_2)$. Here the first inequality uses \eqref{def:q}, \eqref{I:h1>0}, \eqref{I:hn} and \eqref{I:vn}, and the second inequality uses Lemma~\ref{lem:q}. We then integrate it over the interval $[T_3,T_2]$ to show that 
\begin{align*}
\zeta_1(T_2;x_1)\leq&\zeta_1(T_3;x_1)+
(-m(0)+2^{1/\alpha+1}b\epsilon)\frac{C_2}{\epsilon}\int^{T_2}_{T_3}q^{-1-\epsilon-\sigma}(t)~dt\\
\leq&\frac{C_2}{\epsilon}q^{-\epsilon-\sigma}(T_3)-(-m(0)+2^{1/\alpha+1}b\epsilon) \\
&\hspace{70pt}\times\frac{1}{\epsilon+\sigma}\frac{1+\epsilon}{1-\epsilon}\frac{1}{m(0)}
\frac{C_2}{\epsilon}(q^{-\epsilon-\sigma}(T_2)-q^{-\epsilon-\sigma}(T_3)) \\
<&\frac{C_2}{\epsilon}q^{-\epsilon-\sigma}(T_3)
+\frac{1+\epsilon}{\epsilon+\sigma}\frac{1+\epsilon}{1-\epsilon}
\frac{C_2}{\epsilon}(q^{-\epsilon-\sigma}(T_2)-q^{-\epsilon-\sigma}(T_3)) \\
=&\Big(1-\frac{1}{\epsilon+\sigma}\frac{(1+\epsilon)^2}{1-\epsilon}\Big)
\frac{C_2}{\epsilon}q^{-\epsilon-\sigma}(T_3)
+\frac{1}{\epsilon+\sigma}\frac{(1+\epsilon)^2}{1-\epsilon}
\frac{C_2}{\epsilon}q^{-\epsilon-\sigma}(T_2) \\
<&\frac{C_2}{\epsilon}q^{-\epsilon-\sigma}(T_2).
\end{align*}
Therefore \eqref{claim:hn} holds for $n=1$ throughout the interval $[0,T_2]$. Here the second inequality uses \eqref{I:hn} and \eqref{I:s>1h}, the third inequality uses \eqref{A:m1}. Indeed, 
\[
-m(0)>2^{1/\alpha+1}b\qquad \text{if $\epsilon<2^{-1/\alpha-1}b^{-1}$.} 
\]
The last inequality uses \eqref{I:sigma} and Lemma~\ref{lem:q}. Indeed, 
\[
0<\frac{1}{\epsilon+\sigma}\frac{(1+\epsilon)^2}{1-\epsilon}<1
\]
for $\epsilon>0$ sufficiently small.

\subsection*{Proof of \eqref{claim:hn} for $n\geq2$}
We establish one more preliminary.

\begin{lemma}\label{lem:Stirling}
For $n\geq 2$,
\begin{equation}\label{sum:h}
\sum_{j=2}^n\left(\begin{matrix} n\\j\end{matrix}\right)j^{(j-1)/\alpha+1}(n+1-j)^{(n+1-j)/\alpha}
<\frac{2e}{2^{1/\alpha-1}-1}n^{n/\alpha+1}.
\end{equation}
\end{lemma}

The proof uses Stirling's inequality. We include the detail in Appendix~\ref{sec:appendix}.

\

For $n\geq 2$, let $|\zeta_n(T_2;x_n)|=\max_{x\in\mathbb{R}}|\zeta_n(T_2;x)|$. We may assume, without loss of generality, that $\zeta_n(T_2;x_n)>0$. We may choose $T_3$ close to $T_2$ so that  
\begin{equation}\label{I:hn>0}
\zeta_n(t;x_n)\geq0\qquad \text{for all $t\in[T_3,T_2]$}.
\end{equation}
Moreover, we may choose $T_3$ closer to $T_2$, if necessary, so that \eqref{I:s>1h} holds throughout the interval $[T_3,T_2]$. Note from \eqref{e:hn} that
\begin{align*}
\frac{d\zeta_n}{dt}&(t;x_n)\\=&-nv_1(t;x_n)\zeta_n(t;x_n)
-\sum_{j=2}^n\left(\begin{matrix} n\\j\end{matrix}\right)v_j(t;x_n)\zeta_{n+1-j}(t;x_n)-v_{n+1}(t;x_n) \\
\leq&-nm(0)\frac{C_2}{\epsilon}n^{n/\alpha}b^{n-1}q^{-1}(t)q^{-\epsilon-n\sigma}(t)\\
&+\sum_{j=2}^n\left(\begin{matrix} n\\j\end{matrix}\right)\frac{C_2^2}{\epsilon}
j^{(j-1)/\alpha+1}(n+1-j)^{(n+1-j)/\alpha}b^{n-1}q^{-1-(j-1)\sigma}(t)q^{-\epsilon-(n+1-j)\sigma}(t)\\
&+C_2(n+1)^{n/\alpha+1}b^nq^{-1-n\sigma}(t) \\
\leq&-nm(0)\frac{C_2}{\epsilon}n^{n/\alpha}b^{n-1}q^{-1-\epsilon-n\sigma}(t) \\
&+\frac{2e}{2^{1/\alpha-1}-1}n^{n/\alpha+1}\frac{C_2^2}{\epsilon}b^{n-1}q^{-1-\epsilon-n\sigma}(t) 
+\epsilon\frac{C_2}{\epsilon}\Big(\frac{n+1}{n}\Big)^{n/\alpha}(n+1)n^{n/\alpha}b^nq^{-1-n\sigma}(t) \\
\leq&\Big(-m(0)n+\frac{2e}{2^{1/\alpha-1}-1}C_2n+e^{1/\alpha}b\epsilon(n+1)\Big)
\frac{C_2}{\epsilon}n^{n/\alpha}b^{n-1}q^{-1-\epsilon-n\sigma}(t)
\end{align*}
for all $t\in(T_3,T_2)$. Here the first inequality uses \eqref{def:q}, \eqref{I:hn>0}, \eqref{I:hn} and \eqref{I:vn}, the second inequality uses \eqref{sum:h}, and the last inequality uses Lemma~\ref{lem:q}. We then integrate it over the interval $[T_3,T_2]$ to show that 
\begin{align*}
\zeta_n(T_2&;x_n)\\ \leq& \zeta_n(T_3;x_n)\\
&+\Big(-m(0)n+\frac{2e}{2^{1/\alpha-1}-1}C_2n+e^{1/\alpha}b\epsilon(n+1)\Big)
\frac{C_2}{\epsilon}n^{n/\alpha}b^{n-1}\int^{T_2}_{T_3}q^{-1-\epsilon-n\sigma}(t)~dt \\
\leq&\frac{C_2}{\epsilon}n^{n/\alpha}b^{n-1}q^{-\epsilon-n\sigma}(T_3) \\
&-\Big(-m(0)n+\frac{2e}{2^{1/\alpha-1}-1}C_2n+e^{1/\alpha}b\epsilon(n+1)\Big) \\
&\hspace{20pt}\times \frac{1}{\epsilon+n\sigma}\frac{1+\epsilon}{1-\epsilon}\frac{1}{m(0)}
\frac{C_2}{\epsilon}n^{n/\alpha}b^{n-1}(q^{-\epsilon-n\sigma}(T_2)-q^{-\epsilon-n\sigma}(T_3))\\
<&\frac{C_2}{\epsilon}n^{n/\alpha}b^{n-1}q^{-\epsilon-n\sigma}(T_3) \\
&+\frac{n+\epsilon n+\epsilon(n+1)}{\epsilon+n\sigma}\frac{1+\epsilon}{1-\epsilon}
\frac{C_2}{\epsilon}n^{n/\alpha}b^{n-1}(q^{-\epsilon-n\sigma}(T_2)-q^{-\epsilon-n\sigma}(T_3))\\
\leq&\Big(1-\frac{2+5\epsilon}{2\sigma+\epsilon}\frac{1+\epsilon}{1-\epsilon}\Big)
\frac{C_2}{\epsilon}n^{n/\alpha}b^{n-1}q^{-\epsilon-n\sigma}(T_3)
+\frac{2+5\epsilon}{2\sigma+\epsilon}\frac{1+\epsilon}{1-\epsilon}
\frac{C_2}{\epsilon}n^{n/\alpha}b^{n-1}q^{-\epsilon-n\sigma}(T_2) \\
<&\frac{C_2}{\epsilon}n^{n/\alpha}b^{n-1}q^{-\epsilon-n\sigma}(T_2).
\end{align*}
Therefore \eqref{claim:hn} holds for $n=2,3,\dots$ throughout the interval $[0,T_2]$. Here the second inequality uses \eqref{I:hn} and \eqref{I:s>1h}, the third inequality uses \eqref{A:m1} and \eqref{A:m2}. Indeed,
\[
\epsilon(-m(0))^{1/4}>\frac{2e}{2^{1/\alpha-1}-1}
\quad\text{and}\quad -m(0)>e^{1/\alpha}b
\]
if $0<\epsilon<e^{-1/\alpha}b^{-1}$. The fourth inequality uses \eqref{I:sigma} and that ${\displaystyle \frac{(1+2\epsilon)n+\epsilon}{n\sigma+\epsilon}}$ decreases in $n$ for $n\geq 2$, by a direct calculation. The last inequality uses \eqref{I:sigma} and Lemma~\ref{lem:q}. Indeed,
\[
0<\frac{2+5\epsilon}{2\sigma+\epsilon}\frac{1+\epsilon}{1-\epsilon}<1
\]
for $\epsilon>0$ sufficiently small, by a direct calculation.

\subsection*{Proof of \eqref{claim:v0}}
Recall \eqref{e:Kn}. We choose $\delta(t)=q^\sigma(t)$ and use \eqref{I:h0}, \eqref{I:hn} to calculate that 
\begin{equation} \label{I:K0}
|K_0(t;x)|\leq\frac{40}{\pi}\frac1\epsilon\Big(\frac{C_2}{\epsilon}q^{-\sigma\epsilon}(t)q^{-\epsilon}(t)
+\frac{C_2}{\epsilon}q^{\sigma-\sigma\epsilon}(t)q^{-\epsilon-\sigma}(t)\Big) 
=\frac{40}{\pi}\frac{C_2}{\epsilon^2}q^{-\sigma\epsilon-\epsilon}(t)
\end{equation}
for all $t\in[0,T_2]$ for all $x\in\mathbb{R}$. We then integrate \eqref{e:v0} over the interval $[0,T_2]$ to show that
\begin{align*}
|v_0(T_2;x)|\leq&\|u_0\|_{L^\infty}+\int^{T_2}_0|K_0(t;x)|~dt \\
<&\frac12C_0+\frac{40}{\pi}\frac{C_2}{\epsilon^2}\int^{T_2}_0q^{-\sigma\epsilon-\epsilon}(t)~dt \\
\leq&\frac12C_0-\frac{40}{\pi}\frac{C_2}{\epsilon^2}
\frac{1}{1-\sigma\epsilon-\epsilon}\frac{1}{(1-\epsilon)^{1+\sigma\epsilon+\epsilon}}\frac{1}{m(0)}
\Big(\frac{1}{(1-\epsilon)^{1-\sigma\epsilon-\epsilon}}-q^{1-\sigma\epsilon-\epsilon}(T_2)\Big)\\
<&\frac12C_0-\frac{40}{\pi}\frac{1}{1-\sigma\epsilon-\epsilon}\frac{1}{(1-\epsilon)^2}
\frac{1}{\epsilon^2}\frac{C_2}{m(0)}\\
<&\frac12C_0+\frac12(-m(0))\\
<&C_0
\end{align*}
for all $x\in\mathbb{R}$. Therefore \eqref{claim:v0} holds throughout the interval $[0,T_2]$. Here the second inequality uses \eqref{def:C} and \eqref{I:K0}, the third inequality uses \eqref{I:s>1} and that $\sigma\epsilon+\epsilon<1$ for $\epsilon>0$ sufficiently small, the fourth inequality uses Lemma~\ref{lem:q}, and the fifth inequality uses \eqref{A:m3}. Indeed,
\[
\epsilon^2(1-\epsilon)^2(1-\sigma\epsilon-\epsilon)(-m(0))^{5/4}>\frac{80}{\pi}
\]
for $\epsilon>0$ sufficiently small. The last inequality uses \eqref{def:C}.

\subsection*{Proof of \eqref{claim:v1}}
For $n\geq 1$, use \eqref{e:Kn}, where $\delta(t)=n^{-1/\alpha}q^\sigma(t)$, and \eqref{I:hn} to calculate that 
\begin{align}
|K_n(t;x)|\leq&\frac{40}{\pi}\frac1\epsilon
\Big(n^{\epsilon/\alpha}\frac{C_2}{\epsilon}n^{n/\alpha}b^{n-1}
q^{-\sigma\epsilon}(t)q^{-\epsilon-n\sigma}(t) \notag \\
&\hspace{30pt}+n^{\epsilon/\alpha-1/\alpha}\frac{C_2}{\epsilon}(n+1)^{(n+1)/\alpha}b^n
q^{\sigma-\sigma\epsilon}(t)q^{-\epsilon-(n+1)\sigma}(t)\Big) \notag \\
=&\frac{40}{\pi}\Big(1+\Big(\frac{n+1}{n}\Big)^{n/\alpha+1/\alpha}b\Big)
\frac{C_2}{\epsilon^2}n^{n/\alpha+\epsilon/\alpha}b^{n-1}q^{-\sigma\epsilon-\epsilon-n\sigma}(t)\notag \\
<&\frac{40}{\pi}(1+(2e)^{1/\alpha}b)
\frac{C_2}{\epsilon^2}n^{(n-1)/\alpha+2}b^{n-1}q^{-1-\sigma\alpha-(n-1)\sigma}(t)\label{I:Kn}
\end{align}
for all $t\in[0,T_2]$ for all $x\in\mathbb{R}$. Here the last inequality uses \eqref{I:sigma}. Indeed, 
\[
n/\alpha+\epsilon/\alpha<(n-1)/\alpha+2\quad\text{and}\quad
\sigma\epsilon+\epsilon+n\sigma<1+\sigma\alpha+(n-1)\sigma
\] 
for $n\geq 1$ an integer.

Suppose for now that $v_1(T_2;x)\geq0$. Note from \eqref{e:vn} that 
\[
\frac{dv_1}{dt}(t;x)=-v_1^2(t;x)-K_1(t;x)\leq|K_1(t;x)|
\]
for all $t\in(0,T_2)$ for all $x\in\mathbb{R}$. We then integrate it over the interval $[0,T_2]$ to show that
\begin{align*}
v_1(T_2;x)\leq&\|u_0'\|_{L^\infty}+\int^{T_2}_0|K_1(t;x)|~dt \\
\leq&\frac12C_1+\frac{40}{\pi}(1+(2e)^{1/\alpha}b)\frac{C_2}{\epsilon^2}\int^{T_2}_0q^{-2}(t)~dt \\
\leq&\frac12C_1-\frac{40}{\pi}(1+(2e)^{1/\alpha}b)\frac{C_2}{\epsilon^2}
\frac{1}{(1-\epsilon)^3}\frac{1}{m(0)}(q^{-1}(T_2)-(1-\epsilon)) \\
<&\frac12C_1
-\frac{40}{\pi}(1+(2e)^{1/\alpha}b)\frac{C_2}{\epsilon^2}\frac{1}{(1-\epsilon)^3}\frac{1}{m(0)}q^{-1}(T_2)\\
<&\frac12C_1q^{-1}(T_2)-\frac12m(0)q^{-1}T_2 \\
<&C_1q^{-1}(t).
\end{align*}
The second inequality uses \eqref{def:C} and \eqref{I:Kn}, \eqref{I:sigma}, Lemma~\ref{lem:q}. Indeed, $\sigma\alpha<1-10\epsilon$. The third inequality uses \eqref{I:s>1}, and the fifth inequality uses Lemma~\ref{lem:q} and \eqref{A:m3}. Indeed, 
\[
\epsilon^2(1-\epsilon)^3(-m(0))^{5/4}>\frac{80}{\pi}(1+(2e)^{1/\alpha}b)
\]
for $\epsilon>0$ sufficiently small. The last inequality uses \eqref{def:C}. 

Suppose on the contrary that $v_1(T_2;x)<0$. We may assume, without loss of generality, that $\|u_0'\|_{L^\infty}=-m(0)$. We then infer from \eqref{def:q} and \eqref{def:C} that 
\[
v_1(T_2;x)\geq m(T_2)=m(0)q^{-1}(T_2)>-C_1q^{-1}(T_2).
\]
Therefore \eqref{claim:v1} holds throughout the interval $[0,T_2]$.

\subsection*{Proof of \eqref{claim:vn} for $n\geq 3$} We gather some more preliminaries.

For $n\geq 2$, abusing notation, let 
\[
v_1(T_3;x)=m(T_3)\quad\text{and}\quad
v_1(t;x)\leq \frac{1}{(1+\epsilon)^{1/(2+(n-1)\sigma)}}m(t)
\]
for all $t\in[T_3,T_2]$ for some $T_3\in (0,T_2)$ and for some $x\in\mathbb{R}$. Indeed, since $v_1$ and $m$ are uniformly continuous throughout the interval $[0,T_2]$, we may choose $T_3$ sufficiently close to $T_2$ so that the latter inequality holds for all $t\in[T_3,T_2]$. We repeat the argument in the proof of \eqref{I:s>1h} to show that 
\begin{multline}\label{I:s>1u}
\int^{T_2}_{T_3}q^{-2-(n-1)\sigma}(t)~dt\\
\leq-\frac{1}{1+(n-1)\sigma}\frac{1+\epsilon}{1-\epsilon}\frac{1}{m(0)}
(q^{-2-(n-1)\sigma}(T_2)-q^{-2-(n-1)\sigma}(T_3)).
\end{multline}
For $n\geq 3$, moreover, we repeat the argument in the proof of Lemma~\ref{lem:Stirling} to show that 
\begin{equation}\label{sum:u}
\sum_{j=2}^{n-1}\left(\begin{matrix} n\\j\end{matrix}\right)j^{(j-1)/\alpha+1}(n+1-j)^{(n-j)/\alpha+1}
\leq \frac{2e}{2^{1/\alpha-1}-1}n^{(n-1)/\alpha+2}.
\end{equation}

\

For $n\geq 3$, abusing notation, let $|v_n(T_2;x_n)|=\max_{x\in\mathbb{R}}|v_n(T_2;x)|$. We may assume, without loss of generality, that $v_n(T_2;x_n)>0$. Since $v_n$ is uniformly continuous throughout the interval $[0,T_2]$, we may choose $T_3$ sufficiently close to $T_2$ so that 
\begin{equation}\label{I:vn>0}
v_n(t;x_n)\geq 0 \qquad\text{for all $t\in[T_3,T_2]$}
\end{equation}
and \eqref{I:s>1u} holds. Note from \eqref{e:vn} that 
\begin{align*}
\frac{dv_n}{dt}(t;x_n)=&-(n+1)v_1(t;x_n)v_n(t;x_n)
-\sum_{j=2}^{n-1}\left(\begin{matrix} n\\j\end{matrix}\right)v_j(t;x_n)v_{n+1-j}(t;x_n)-K_n(t;x_n) \\
\leq&-(n+1)m(0)C_2n^{(n-1)/\alpha+1}b^{n-1}q^{-1}(t)q^{-1-(n-1)\sigma}(t)\\
&+\sum_{j=2}^{n-1}\left(\begin{matrix} n\\j\end{matrix}\right)C_2^2
j^{(j-1)/\alpha+1}(n+1-j)^{(n-j)/\alpha+1}b^{n-1}q^{-1-(j-1)\sigma}(t)q^{-1-(n-j)\sigma}(t)\\
&+|K_n(t;x_n)| \\
\leq&-(n+1)m(0)C_2n^{(n-1)/\alpha+1}b^{n-1}q^{-2-(n-1)\sigma}(t) \\
&+C_2^2\frac{2e}{2^{1/\alpha-1}-1}n^{(n-1)/\alpha+2}b^{n-1}q^{-2-(n-1)\sigma}(t)\\
&+\frac{40}{\pi}(1+(2\epsilon)^{1/\alpha}b)\frac{C_2}{\epsilon^2}
n^{(n-1)/\alpha+2}b^{n-1}q^{-1-\sigma\alpha-(n-1)\sigma}(t)\\
\leq&\Big(-m(0)(n+1)+\frac{2e}{2^{1/\alpha-1}-1}C_2n
+\frac{40}{\pi}(1+(2\epsilon)^{1/\alpha}b)\frac{1}{\epsilon^2}n\Big)\\
&\hspace*{150pt}\times C_2n^{(n-1)/\alpha+1}b^{n-1}q^{-2-(n-1)\sigma}(t)
\end{align*}
for all $t\in(T_3,T_2)$. Here the first inequality uses \eqref{def:q}, \eqref{I:vn>0} and \eqref{I:vn}, the second inequality uses \eqref{sum:u} and \eqref{I:Kn}, and the last inequality uses Lemma~\ref{lem:q} and \eqref{I:sigma}. Indeed, $\sigma\alpha<1-10\epsilon$.
We then integrate it over the interval $[T_3,T_2]$ to show that 
\begin{align*}
v_n(T_2;x_n)\leq&v_n(T_3;x_n)\\&+
\Big(-m(0)(n+1)+\frac{2e}{2^{1/\alpha-1}-1}C_2n
+\frac{40}{\pi}(1+(2\epsilon)^{1/\alpha}b)\frac{1}{\epsilon^2}n\Big)\\
&\hspace*{120pt}\times C_2n^{(n-1)/\alpha+1}b^{n-1}\int^{T_2}_{T_3}q^{-2-(n-1)\sigma}(t)~dt\\
\leq& C_2n^{(n-1)/\alpha+1}b^{n-1}q^{-1-(n-1)\sigma}(T_3)\\
&-\Big(-m(0)(n+1)+\frac{2e}{2^{1/\alpha-1}-1}C_2n
+\frac{40}{\pi}(1+(2\epsilon)^{1/\alpha}b)\frac{1}{\epsilon^2}n\Big)\\
&\hspace{83pt}\times\frac{1}{1+(n-1)\sigma}\frac{1+\epsilon}{1-\epsilon}\frac{1}{m(0)}
C_2n^{(n-1)/\alpha+1}b^{n-1}\\ &\hspace{150pt}\times(q^{-1-(n-1)\sigma}(T_2)-q^{-1-(n-1)\sigma}(T_3))\\
<&C_2n^{(n-1)/\alpha+1}b^{n-1}q^{-1-(n-1)\sigma}(T_3)\\
&+\frac{n+1+\epsilon n}{1+(n-1)\sigma}\frac{1+\epsilon}{1-\epsilon}
C_2n^{(n-1)/\alpha+1}b^{n-1}(q^{-1-(n-1)\sigma}(T_2)-q^{-1-(n-1)\sigma}(T_3))\\
<&C_2n^{(n-1)/\alpha+1}b^{n-1}q^{-1-(n-1)\sigma}(T_2).
\end{align*}
Therefore \eqref{claim:vn} holds for $n=3,4,\dots$ throughout the interval $[0,T_2]$. Here the second inequality uses \eqref{I:vn} and \eqref{I:s>1u}, the third inequality uses \eqref{A:m2} and \eqref{A:m3}. Indeed, 
\[
-\epsilon m(0)>\frac{2e}{2^{1/\alpha-1}-1}C_2+\frac{40}{\pi}(1+(2e)^{1/\alpha}b)\frac{1}{\epsilon^2}
\]
for $\epsilon>0$ sufficiently small. The last inequality uses that ${\displaystyle \frac{(1+\epsilon)n+1}{n\sigma+1-\sigma}}$ decreases in $n\geq 3$, by a direct calculation, and \eqref{I:sigma}, Lemma~\ref{lem:q}. Indeed,
\[
0<\frac{4+3\epsilon}{2\sigma+1}\frac{1+\epsilon}{1-\epsilon}<1
\]
for $\epsilon>0$ sufficiently small.

\subsection*{Proof of \eqref{claim:vn} for $n=2$}
Abusing notation, let $|v_2(T_2;x_2)|=\max_{x\in\mathbb{R}}|v_2(T_2;x)|$. We may assume, without loss of generality, that $v_2(T_2;x_2)>0$. We may choose $T_3$ close to $T_2$ so that 
\begin{equation}\label{I:v2>0}
v_2(t;x_2)\geq 0 \qquad\text{for all $t\in[T_3,T_2]$}.
\end{equation}
Moreover, we may choose $T_3$ closer to $T_2$, if necessary, so that \eqref{I:s>1u} holds.

Suppose for now that $x_2\not\in \Sigma_{1/3}(T_2)$. That is, $v_1(T_2;x_2)>\frac23m(T_2)$ (see \eqref{def:Sigma}). We may choose $T_3$ closer to $T_2$, if necessary, so that 
\begin{equation}\label{I:1/3}
v_1(t;x_2)\geq \frac23m(t)\qquad\text{for all $t\in[T_3,T_2]$}.
\end{equation}
Note from \eqref{e:vn} that 
\begin{align*}
\frac{dv_2}{dt}(t;x_2)=&-3v_1(t;x_2)v_2(t;x_2)-K_2(t;x_2) \\
\leq&-2m(0)C_22^{1/\alpha+1}bq^{-1}(t)q^{-1-\sigma}(t) \\
&+\frac{40}{\pi}(1+(2e)^{1/\alpha}b)\frac{C_2}{\epsilon^2}
2^{1/\alpha+2}bq^{-1-\sigma\alpha-\sigma}(t)\\
\leq&2\Big(-m(0)+\frac{40}{\pi}(1+(2e)^{1/\alpha}b)\frac{1}{\epsilon^2}\Big)
C_22^{1/\alpha+1}bq^{-2-\sigma}(t)
\end{align*}
for all $t\in(T_3,T_2)$. Here the first inequality uses \eqref{I:1/3}, \eqref{I:v2>0}, \eqref{def:q} and \eqref{I:Kn}, and the second inequality uses \eqref{I:sigma} and Lemma~\ref{lem:q}. Indeed, $\sigma\alpha<1-10\epsilon$. 
We then integrate it over the interval $[T_3,T_2]$ to show that 
\begin{align*}
v_2(T_2;x_2)\leq&v_2(T_3;x_2)\\ &+2\Big(-m(0)+\frac{40}{\pi}(1+(2e)^{1/\alpha}b)\frac{1}{\epsilon^2}\Big)
C_22^{1/\alpha+1}b\int^{T_2}_{T_3}q^{-2-\sigma}(t)~dt \\
\leq&C_22^{1/\alpha+1}bq^{-1-\sigma}(T_3) \\
&-2\Big(-m(0)+\frac{40}{\pi}(1+(2e)^{1/\alpha}b)\frac{1}{\epsilon^2}\Big) \\
&\hspace*{57pt}\times\frac{1}{1+\sigma}\frac{1+\epsilon}{1-\epsilon}\frac{1}{m(0)}
C_22^{1/\alpha+1}b(q^{-1-\sigma}(T_2)-q^{-1-\sigma}(T_3))\\
<&C_22^{1+1/\alpha}bq^{-1-\sigma}(T_3)
+\frac{2}{1+\epsilon}\frac{(1+\epsilon)^2}{1-\epsilon}
C_22^{1/\alpha+1}b(q^{-1-\sigma}(T_2)-q^{-1-\sigma}(T_3))\\
<&C_22^{1+1/\alpha}bq^{-1-\sigma}(T_2).
\end{align*}
Here the second inequality uses \eqref{I:vn} and \eqref{I:s>1u}, and the third inequality uses \eqref{A:m3}. Indeed,
\[
-\epsilon^3m(0)>\frac{40}{\pi}(1+(2e)^{1/\alpha}b)
\]
for $\epsilon>0$ sufficiently small. The last inequality uses \eqref{I:sigma} and Lemma~\ref{lem:q}. Indeed,
\[
0<\frac{2}{1+\sigma}\frac{(1+\epsilon)^2}{1-\epsilon}<1
\]
for $\epsilon>0$ sufficiently small.

\

Suppose on the contrary that $x_2\in\Sigma_{1/3}(T_2)$. Lemma~\ref{lem:S} then dictates that 
\begin{equation}\label{I:2/3}
v_1(t;x_2)\leq\frac23m(t)<0\qquad \text{for all $t\in[0,T_2]$.}
\end{equation}
Differentiating \eqref{def:X} with respect to $x$ and using \eqref{def:vh_n}, we arrive at 
\begin{alignat}{2}
\frac{d}{dt}(\partial_xX)=&v_1(\partial_xX), \qquad &&(\partial_xX)(0;x)=1\label{e:X1}
\intertext{and}
\frac{d}{dt}(\partial_x^2X)=&v_2(\partial_xX)^2+v_1(\partial_x^2X), 
\qquad &&(\partial_x^2X)(0;x)=0, \label{e:X2} \\
\frac{d}{dt}(\partial_x^3X)=&v_3(\partial_xX)^3
+3v_2(\partial_xX)(\partial_x^2X)+v_1(\partial_x^3X), \qquad &&(\partial_x^3X)(0;x)=0.\label{e:X3}
\end{alignat}
An integration of \eqref{e:v0} leads to that 
\[
v_0(t;x)=\phi(x)-\int^t_0K_0(t;x)~dt.
\]
Differentiating it with respect to $x$ and using \eqref{def:vh_n}, we arrive at
\begin{align}
(v_2(\partial_xX)^2+v_1(\partial_x^2X))(t;x)\qquad \qquad=u_0''(x)-I_2(t;x),\label{e:v2I}\\
(v_3(\partial_xX)^3+3v_2(\partial_xX)(\partial_x^2X)+v_1(\partial_x^3X))(t;x)=u_0'''(x)-I_3(t;x), \label{e:v3I}
\end{align}
where
\begin{align}
I_2(t;x)=&\int^t_0(K_2(\partial_xX)^2+K_1(\partial_x^2X))(\tau;x)~d\tau,\label{def:I2} \\
I_3(t;x)=&\int^t_0(K_3(\partial_xX)^3+3K_2(\partial_xX)(\partial_x^2X)+K_1(\partial_x^3X))(\tau;x)~d\tau.
\label{def:I3}
\end{align}
Moreover, note from \eqref{e:X3} and \eqref{e:v3I} that 
\begin{equation}\label{e:X3'}
\frac{d}{dt}(\partial_x^3X)(\cdot\;;x)=u_0'''(x)-I_3(\cdot\;;x),\qquad (\partial_x^3X)(0;x)=0.
\end{equation}

We claim that 
\begin{equation}\label{I:X1}
\frac12q^{1+2\epsilon}(t)\leq (\partial_xX)(t;x_2)\leq 2q^{1-\epsilon}(t)\qquad\text{for all $t\in[0,T_2]$}.
\end{equation}
Indeed, note from \eqref{def:X}, \eqref{e:X1} and \eqref{def:r}, \eqref{I:dr/dt} that 
\[
\frac{1}{1-\epsilon}\frac{dr/dt}{r}\leq 
\frac{d(\partial_xX)/dt}{\partial_xX}\leq\frac{1}{1+\epsilon}\frac{dr/dt}{r}
\]
throughout the interval $(0,T_2)$. We then integrate it and use \eqref{e:X1} to show that 
\[
\Big(\frac{r(t)}{r(0)}\Big)^{1/(1-\epsilon)}\leq (\partial_xX)(t;x_2)\leq \Big(\frac{r(t)}{r(0)}\Big)^{1/(1+\epsilon)}.
\]
The claim therefore follows from \eqref{I:qr}. 

To proceed, we claim that
\begin{align}
|(\partial_x^2X)(t;x_2)|<&-\frac{2^{1/\alpha+4}}{m(0)}C_2bq^{2-\sigma-2\epsilon}(t) \label{claim:X2}
\intertext{and}
|(\partial_x^3X)(t;x_2)|<&\frac{\epsilon}{m^2(0)}C_2^2b^2q^{3-2\sigma+7\epsilon}(t)\label{claim:X3}
\end{align}
for all $t\in[0,T_2]$. 
It follows from \eqref{e:X2} and \eqref{e:X3} that \eqref{claim:X2} and \eqref{claim:X3} hold at $t=0$.  Suppose on the contrary that \eqref{claim:X2} and \eqref{claim:X3} hold throughout the interval $[0,T_4)$, but one of them fails at $t=T_4$ for some $T_4\in(0,T_2]$. By continuity, we may assume that 
\begin{align}
|(\partial_x^2X)(t;x_2)|\leq &-\frac{2^{1/\alpha+4}}{m(0)}C_2bq^{2-\sigma-2\epsilon}(t) \label{I:X2}
\intertext{and}
|(\partial_x^3X)(t;x_2)|\leq &\frac{\epsilon}{m^2(0)}C_2^2b^2q^{3-2\sigma+7\epsilon}(t) \label{I:X3}
\end{align}
for all $t\in[0,T_4]$. We seek a contradiction.

\subsection*{Proof of \eqref{claim:X2}}We recall \eqref{def:I2} and calculate that 
\begin{align}
|I_2(t;x_2)|\leq& \frac{40}{\pi}(1+(2e)^{1/\alpha}b)\frac{C_2}{\epsilon^2}
\int^t_0\Big(4\cdot2^{1/\alpha+2}b
q^{-1-\sigma\alpha-\sigma}(\tau)q^{2-2\epsilon}(\tau) \\
&\hspace*{110pt}-\frac{2^{1/\alpha+4}}{m(0)}C_2b
q^{-1-\sigma\alpha}(\tau)q^{2-\sigma-2\epsilon}(\tau)\Big)~d\tau \notag\\
\leq&\frac{40}{\pi}(1+(2e)^{1/\alpha}b)\frac{C_2}{\epsilon^2}2^{1/\alpha+4}b\Big(1-\frac{C_2}{m(0)}\Big)
\int^t_0 q^{-\sigma+8\epsilon}(\tau)~d\tau\notag\\
\leq&-\frac{40}{\pi}(1+(2e)^{1/\alpha}b)\frac{C_2}{\epsilon^2}2^{1/\alpha+5}b\notag \\
&\quad\times\frac{1}{\sigma-1-8\epsilon}\frac{1}{(1-\epsilon)^{\sigma+1-8\epsilon}}\frac{1}{m(0)}
(q^{1-\sigma+8\epsilon}(t)-(1-\epsilon)^{\sigma-1-8\epsilon}) \notag\\
<&\epsilon C_2bq^{1-\sigma+8\epsilon}(t) \label{I:I2}
\end{align}
for all $t\in[0,T_4]$. Here the first inequality uses \eqref{I:Kn}, \eqref{I:X1} and \eqref{I:X2}, and the second inequality uses Lemma~\ref{lem:q} and \eqref{I:sigma}. Indeed, 
\[
\sigma+\sigma\alpha+2\epsilon-1<\sigma-8\epsilon.
\] 
The third inequality uses \eqref{def:C} and \eqref{I:s>1}, and the last inequality uses Lemma~\ref{lem:q}, \eqref{I:sigma} and \eqref{A:m3}. Indeed,
\[
-\epsilon^3(1-\epsilon)^{\sigma+1-8\epsilon}m(0)>
\frac{1}{\sigma-1-8\epsilon}\frac{80}{\pi}(1+(2e)^{1/\alpha}b)2^{1/\alpha+5}
\]
for $\epsilon>0$ sufficiently small.
We then evaluate \eqref{e:v2I} at $t=T_4$ and $x=x_2$ to show that 
\begin{align*}
|(\partial_x^2X)&(T_4;x_2)|\\
=&|v_1^{-1}(T_4;x_2)||u_0''(x_2)-I_2(T_4;x_2)-v_2(T_4;x_2)(\partial_xX)(T_4;x_2)^2| \\
<&-\frac32\frac{1}{m(0)}q(T_4)(2^{1/\alpha+1}b+\epsilon C_2bq^{1-\sigma+8\epsilon}(T_4)
+4\cdot2^{1/\alpha+1}C_2bq^{-1-\sigma}(T_4)q^{2-2\epsilon}(T_4)) \\
\leq&-\frac32(5\cdot2^{1/\alpha+1}+\epsilon)\frac{1}{m(0)}C_2bq^{2-\sigma-2\epsilon}(T_4)\\
<&-\frac{2^{1/\alpha+4}}{m(0)}C_2b^{1/\alpha}q^{2-\sigma-2\epsilon}(T_4).
\end{align*}
Therefore \eqref{claim:X2} holds throughout the interval $[0,T_2]$. Here the first inequality uses \eqref{I:2/3}, \eqref{def:q} and \eqref{A:un}, \eqref{I:I2}, \eqref{I:vn}, \eqref{I:X1}, the second inequality uses \eqref{def:C} and \eqref{I:sigma}, Lemma~\ref{lem:q}, and the last inequality follows for $\epsilon>0$ sufficiently small. 

\subsection*{Proof of \eqref{claim:X3}}Similarly, recall \eqref{def:I3} and we calculate that
\begin{align}
|I_3(t;x_2)|<&\frac{40}{\pi}(1+(2e)^{1/\alpha}b)\frac{C_2}{\epsilon^2}\notag \\
&\times\int^t_0\Big(8\cdot3^{2/\alpha+2}b^2
q^{-1-\sigma\alpha-2\sigma}(\tau)q^{3-3\epsilon}(\tau) \notag\\ 
&\hspace*{30pt} -6\cdot2^{2/\alpha+6}\frac{C_2}{m(0)}b^2
q^{-1-\sigma\alpha-\sigma}(\tau)q^{1-\epsilon}(\tau)q^{2-\sigma-2\epsilon}(\tau) \notag \\
&\hspace*{120pt}+\frac{\epsilon}{m^2(0)}C_2^2b^2
q^{-1-\sigma\alpha}(\tau)q^{3-2\sigma+7\epsilon}(\tau)\Big)~d\tau \notag\\
\leq&\frac{40}{\pi}(1+(2e)^{1/\alpha}b) \notag \\
&\times\Big(2^33^{2/\alpha+2}
+3\cdot2^{2/\alpha+7}\frac{C_2}{m(0)}+\frac{\epsilon C_2^2}{m^2(0)}\Big)
\frac{C_2}{\epsilon^2}b^2\int^t_0q^{1-2\sigma+7\epsilon}(\tau)~d\tau \notag\\
\leq&\frac{40}{\pi}(1+(2e)^{1/\alpha}b)
\Big(2^33^{2/\alpha+2}+\frac{3\cdot2^{2/\alpha+7}}{(-m(0))^{1/4}}+\frac{\epsilon}{(-m(0))^{1/2}}\Big)\notag \\
&\times \frac{1}{2\sigma-2-7\epsilon}\frac{1}{(1-\epsilon)^{2\sigma-7\epsilon}}\frac{1}{m(0)}
\frac{C_2}{\epsilon^2}b^2(q^{2-2\sigma+7\epsilon}(t)-(1-\epsilon)^{2\sigma-2-7\epsilon}) \notag \\
<&-\frac{\epsilon^2}{m(0)}C_2^2b^2q^{2-2\sigma+7\epsilon}(t) \label{I:I3}
\end{align}
for all $t\in[0,T_4]$. Here the first inequality uses \eqref{I:Kn}, \eqref{I:X1}, \eqref{I:X2} and \eqref{I:X3}, and the second inequality uses \eqref{I:sigma}. Indeed, 
\[
2-\sigma\alpha-2\sigma-3\epsilon>1-2\sigma+7\epsilon.
\] 
The third inequality uses \eqref{I:s>1}, and the last inequality uses \eqref{A:m1} and \eqref{A:m3}. Indeed,
\begin{align*}
\epsilon^4(1-\epsilon)^{2\sigma-7\epsilon}(-m(0))^{3/4}>&
\frac{1}{2\sigma-2-7\epsilon}\frac{40}{\pi}(1+(2e)^{1/\alpha}b) 
(2^3\cdot 3^{2/\alpha+2}+3\cdot2^{2/\alpha+7}+\epsilon)
\end{align*}
for $\epsilon>0$ sufficiently small.
We then integrate \eqref{e:X3'} over the the interval $[0,T_4]$ to show that 
\begin{align*}
|(\partial_x^3X)(T_4;x_2)|\leq&\int^{T_4}_0(|u_0'''(x_2)|+|I_3(t;x_2)|)~dt \\
<&\int^{T_4}_0\Big(3^{2/\alpha+1}b^2
-\frac{\epsilon^2}{m(0)}C_2^2b^2q^{2-2\sigma+7\epsilon}(t)\Big)~dt \\
\leq&\Big(\frac{3^{2/\alpha+1}}{C_2^2}-\frac{\epsilon^2}{m(0)}\Big)
\frac{1}{2\sigma-3-7\epsilon}\frac{1}{(1-\epsilon)^{2\sigma-1-7\epsilon}}\frac{1}{m(0)} \\
&\hspace*{80pt}\times C_2^2b^2
(q^{3-2\sigma+7\epsilon}(T_4)-(1-\epsilon)^{2\sigma-3-7\epsilon})\\
<&\frac{\epsilon}{m^2(0)}C_2^2b^2q^{3-2\sigma+7\epsilon}(T_4).
\end{align*}
Therefore \eqref{claim:X3} holds throughout the interval $[0,T_2]$. Here the second inequality uses \eqref{A:un} and \eqref{I:I3}, the third inequality uses \eqref{I:s>1}, and the last inequality uses \eqref{I:sigma} and \eqref{A:m2}. Indeed,
\[
5\epsilon^2(1-\epsilon)^{2\sigma-1-7\epsilon}-1)(3^{2/\alpha+1}(-m(0))^{1/2}+\epsilon^2)>1
\]
for $\epsilon>0$ sufficiently small. This proves \eqref{claim:X2} and \eqref{claim:X3}.

\

Returning to the proof of \eqref{claim:vn} for $n=2$, we recall that $v_2(T_2;x_2)=\max_{x\in\mathbb{R}}|v_2(T_2;x)|$ and $x_2\in\Sigma_{1/3}(T_2)$. Differentiating $v_2$ and evaluating at $t=T_2$, $x=x_2$, we use \eqref{def:vh_n} to find that 
\[
v_3(T_2;x_2)(\partial_xX)(T_2;x_2)=0.
\] 
Let's multiply \eqref{e:v2I} by $3v_2(\partial_xX)$ and \eqref{e:v3I} by $v_1$ and take their difference. Evaluating the result at $t=T_2$ and $x=x_2$, we show that 
\begin{align*}
v_2^2(T_2;x_2)=&\frac13(\partial_xX)^{-3}(T_2;x_2)
(v_1^2(T_2;x_2)(\partial_x^3X)(T_2;x_2) \\
&\hspace*{80pt}+3v_2(T_2;x_2)(\partial_xX)(T_2;x_2)(u_0''(x_2)-I_2(T_2;x_2)) \\
&\hspace*{80pt}-v_1(T_2;x_2)(u_0'''(x_2)-I_3(T_2;x_2))) \\
<&\frac83q^{-3-6\epsilon}(T_2)
\Big(m^2(0)\frac{\epsilon}{m^2(0)}C_2^2b^2q^{-2}(T_2)q^{3-2\sigma+7\epsilon}(T_2) \\
&\hspace*{60pt}+62^{1/\alpha+1}C_2bq^{-1-\sigma}(T_2)q^{1-\epsilon}(T_2)
(2^{1/\alpha+1}b+\epsilon C_2bq^{1-\sigma+8\epsilon}(T_2)) \\
&\hspace*{60pt}-m(0)q^{-1}(T_2)
\Big(3^{2/\alpha+1}b^2-\frac{\epsilon^2}{m(0)}C_2^2b^{2/\alpha}q^{2-2\sigma+7\epsilon}(T_2)\Big)\Big) \\
<&\frac83\Big(\epsilon+6\cdot2^{1/\alpha+1}\Big(\frac{2^{1/\alpha+1}}{C_2}+\epsilon\Big)
-\frac{3^{2/\alpha+1}}{m(0)}+\epsilon^2\Big)C_2^2b^2q^{-2-2\sigma-\epsilon}(T_2)\\
<&C_2^22^{2/\alpha+2}b^2q^{-2-2\sigma}(T_2).
\end{align*}
Therefore \eqref{claim:vn} holds for $n=2$ throughout the interval $[0,T_2]$. Here the first inequality uses \eqref{I:X1}, \eqref{I:2/3}, \eqref{def:q}, \eqref{I:X3}, \eqref{I:vn} and \eqref{A:un}, \eqref{I:I2}, \eqref{I:I3}, the second inequality uses \eqref{def:C} and \eqref{I:sigma}, Lemma~\ref{lem:q}, and the last inequality uses that
\[
\epsilon+6\cdot 2^{1/\alpha+1}(2^{1/\alpha+1}\epsilon^{3/4}+\epsilon)+3^{2/\alpha+1}\epsilon+\epsilon^2<3\cdot 2^{2/\alpha-1}
\]
for $\epsilon>0$ sufficiently small.

\subsection*{Proof of \eqref{claim:v0}-\eqref{claim:vn}, \eqref{claim:h0}-\eqref{claim:hn}}
To summarize, a contradiction proves that \eqref{claim:v0}-\eqref{claim:vn} and \eqref{claim:h0}-\eqref{claim:hn} hold for all $n=0,1,2,\dots$ throughout the interval $[0,T_1]$. 

\subsection*{Proof of \eqref{claim:A}}
Note from \eqref{I:Kn}, \eqref{def:q} and \eqref{A:m3} that 
\[
|K_1(t;x)|\leq \frac{40}{\pi}(1+(2e)^{1/\alpha}b)\frac{C_2}{\epsilon^2}m^{-2}(0)m^2(t)<\epsilon^2m^2(t)
\]
for all $t\in [0,T_1]$ for all $x\in\mathbb{R}$. Indeed,
\[
\epsilon^4(-m(0))^{5/4}>\frac{40}{\pi}(1+(2e)^{1/\alpha}b)
\]
for $\epsilon>0$ sufficiently small. A contradiction therefore proves \eqref{claim:A}. We merely pause to remark that \eqref{claim:v0}-\eqref{claim:vn} and \eqref{claim:h0}-\eqref{claim:hn} hold for all $n=0,1,2,\dots$ throughout the interval $[0, T']$ for all $T'<T$.

\subsection*{Proof of Theorem~\ref{thm:main}}
For $t\in [0,T)$, let $x \in \Sigma_\epsilon(t)$. Note from \eqref{def:r} and \eqref{I:dr/dt} that 
\[
m(0)(v_1^{-1}(0;x)+(1+\epsilon) t) \leq r(t;x)\leq m(0)(v_1^{-1}(0;x)+(1-\epsilon) t).
\]
Moreover, note from Lemma~\ref{lem:S} that 
\[
m(0)<v_1(0;x)\leq (1-\epsilon)m(0).
\] 
Consequently,
\[
1+m(0)(1+\epsilon)t \leq r(t) \leq \frac{1}{1-\epsilon} + m(0) (1-\epsilon)t.
\]
Furthermore, \eqref{I:qr} implies that 
\[
(1-\epsilon) + m(0)(1-\epsilon^2)t \leq q(t) \leq \frac{1}{1-\epsilon} + m(0) (1-\epsilon)t.
\]
Since the left side decreases to zero as ${\displaystyle t\to -\frac{1}{m(0)}\frac{1}{1+\epsilon}}$ and the right side decreases to zero as ${\displaystyle t\to -\frac{1}{m(0)}\frac{1}{(1-\epsilon)^2}}$, it follows that $q(t)\to 0$ and, hence,  $m(t)\to-\infty$ (see \eqref{def:q}) as $t\to T-$, where $T$ satisfies \eqref{E:T}. Note on the other hand that \eqref{claim:v0} dictates that $v_0(t;x)$ remains bounded for all $t\in [0,T']$, $T'<T$, for all $x\in \mathbb{R}$. That is, $\inf_{x\in\mathbb{R}}\partial_xu(x,t)\to -\infty$ as $t\to T-$ but $u(x,t)$ is bounded for all $x\in\mathbb{R}$ for all $t\in [0,T)$, namely wave breaking. This completes the proof.

\section*{Acknowledgment}
VMH is supported by the National Science Foundation grant CAREER DMS-1352597, an Alfred P. Sloan Foundation fellowship, and a Beckman fellowship at the Center for Advanced Study at the University of Illinois at Urbana-Champaign. The authors thank anonymous referees for their careful reading of the manuscript and many helpful suggestions.

\begin{appendix}

\section{Assorted proofs of lemmas}\label{sec:appendix}

\begin{proof}[Proof of Lemma~\ref{lem:Kn}]
We split the integral and perform an integration by parts to show that 
\begin{align}
|K_n&(t;x)| \notag \\=&\Big|-\frac12\Big(\int_{|y|<\delta}+\int_{|y|>\delta}\Big)
\text{csch}(\tfrac{\pi}{2}y)((\partial_x^n\eta)(X(t;x),t)-(\partial_x^n\eta)(X(t;x)-y,t))~dy\Big|\hspace*{-25pt}\notag \\
\leq&\Big|\frac1\pi\log|\tanh(\tfrac\pi4\delta)|
((\partial_x^n\eta)(X(t;x)-\delta,t)-(\partial_x^n\eta)(X(t;x)+\delta,t))\Big|\notag \\
&+\Big|\frac1\pi\int_{|y|<\delta}\log|\tanh(\tfrac\pi4y)|(\partial_x^{n+1}\eta)(X(t;x)-y,t)~dy\Big|\notag\\
&+\Big|\frac12\int_{|y|>\delta}\text{csch}(\tfrac\pi2y)
((\partial_x^n\eta)(X(t;x),t)-(\partial_x^n\eta)(X(t;x)-y,t))~dy\Big|\notag\\
\leq&\frac6\pi\log(\coth(\tfrac\pi4\delta))\|\zeta_n(t)\|_{L^\infty}
+\frac2\pi\Big(\int_0^\delta \log(\coth(\tfrac\pi4y))~dy\Big)\|\zeta_{n+1}(t)\|_{L^\infty}.\label{I:K}
\end{align}
Indeed, $(\frac1\pi\log(\tanh(\frac\pi4y)))'=\frac12\text{csch}(\frac\pi2y)$. In other words, the kernel associated with the integral representation of $\mathcal{M}\partial_x$ is singular of a logarithmic order near zero and it decays exponentially at infinity.

Let $g(y)=\frac1\pi\log(\coth(\tfrac\pi4y))$. A direct calculation reveals that its inverse function is 
\[
g^{-1}(y)=\frac2\pi\log(\coth(\tfrac\pi4y))=2g(2y).
\]
Note that 
\[
\int^\delta_0g(y)~dy=\delta g(\delta)+\int^\infty_{g(\delta)}g^{-1}(y)~dy
=\delta g(\delta)+\int^\infty_{g(\delta)}2g(2y)~dy,
\]
and we calculate that
\begin{align*}
\int^\infty_{g(\delta)}2g(2y)~dy=\int^\infty_{2g(\delta)}g(z)~dz
=&\frac1\pi\int^\infty_{2g(\delta)}\log\Big(1+\frac{2}{e^{\pi z/2}-1}\Big)~dz\\
<&\frac2\pi\int^\infty_{2g(\delta)}\frac{e^{\pi z/2}}{e^{\pi z/2}-1}e^{-\pi z/2}~dz \\
<&\frac2\pi\frac{e^{\pi g(\delta)}}{e^{\pi g(\delta)}-1}\int^\infty_{2g(\delta)}e^{-\pi z/2}~dz \\
<&\frac6\pi\frac2\pi e^{-\pi g(\delta)}=\frac{12}{\pi^2}\tanh(\tfrac\pi4\delta).
\end{align*}
Here the first inequality uses that $\log(1+x)<x$ for all $x>0$, the second inequality uses that ${\displaystyle \frac{e^x}{e^x-1}}$ is decreasing for all $x>0$, and the third inequality uses that ${\displaystyle \frac{\coth(\tfrac\pi4\delta)}{\coth(\tfrac\pi4\delta)-1}<3}$ for $0<\delta<1$, by direct calculations. Consequently,
\begin{align}
\int^\delta_0g(y)~dy<&
\frac1\pi\delta\log(\coth(\tfrac\pi4\delta))+\frac{12}{\pi^2}\tanh(\tfrac\pi4\delta)\notag\\
<&\frac1\pi\delta\log(\coth(\tfrac\pi4\delta))+\frac9\pi\delta\log(\coth(\tfrac\pi4\delta)).\label{I:2'}
\end{align}
Indeed, a direct calculation reveals that $\tanh x<3 x\log(\coth x))$ for all $0<x<1$. 

Substituting \eqref{I:2'} into \eqref{I:K}, we then show that
\begin{align*}
|K_n(t;x)|<& \frac6\pi\log(\coth(\tfrac\pi4\delta))\|\zeta_n(t)\|_{L^\infty}
+\frac{20}{\pi}\delta\log(\coth(\tfrac\pi4\delta))\|\zeta_{n+1}(t)\|_{L^\infty} \\
<&\frac{20}{\pi}(\log(\tfrac\pi8\delta)^{-1}\|\zeta_n(t)\|_{L^\infty}
+\delta\log(\tfrac\pi8\delta)^{-1}\|\zeta_{n+1}(t)\|_{L^\infty}) \\
<&\frac{20}{\pi}\frac1\epsilon\Big(\frac8\pi\Big)^\epsilon
(\delta^\epsilon\|\zeta_n(t)\|_{L^\infty}+\delta^{1-\epsilon}\|\zeta_{n+1}(t)\|_{L^\infty})
\end{align*}
for all $0<\epsilon<1$. Therefore \eqref{e:Kn} holds for $\epsilon>0$ sufficiently small. Here the second inequality uses that $\coth x<2x^{-1}$ for all $0<x<1$, and the last inequality uses that $\log(x^{-1})<\frac1\epsilon x^{-\epsilon}$ throughout $0<x<1$ for all $0<\epsilon<1$, by direct calculations. This completes the proof.
\end{proof}

\begin{proof}[Proof of Lemma~\ref{lem:S}]
Suppose on the contrary that $x_1\notin\Sigma_{\gamma}(t_1)$ but $x_1\in\Sigma_{\gamma}(t_2)$ for some $x_1\in\mathbb{R}$ for some $0\leq t_1\leq t_2\leq T_1$. That is,
\begin{equation}\label{def:x1}
v_1(t_1;x_1)>(1-\gamma)m(t_1)\quad\text{and}\quad
v_1(t_2;x_1)\leq(1-\gamma)m(t_2)<\frac12m(t_2).
\end{equation}
Since $v_1(\cdot\,;x_1)$ and $m$ are uniformly continuous throughout the interval $[0,T_1]$, we may choose $t_1$ and $t_2$ close so that 
\[
v_1(t;x_1)\leq\frac12m(t)\qquad\text{for all $t\in[t_1,t_2]$.}
\]
Let 
\begin{equation}\label{def:x2}
v_1(t_1;x_2)=m(t_1)<\frac12m(t_1).
\end{equation}
We may choose $t_2$ closer to $t_1$, if necessary, so that 
\[
v_1(t;x_2)\leq\frac12m(t)\qquad \text{for all $t\in[t_1,t_2]$.}
\]
For $\epsilon>0$ sufficiently small, \eqref{I:A} then leads to that 
\[
|K_1(t;x_j)|\leq \epsilon^2m^2(t)\leq 4\epsilon^2v_1^2(t;x_j)<\frac\gamma2 v_1^2(t;x_j)
\qquad \text{for all $t\in[t_1,t_2]$ and $j=1,2$.}
\]
Note from \eqref{e:vn}, where $n=1$, that 
\[
\frac{dv_1}{dt}(\cdot\,;x_1)=-v_1^2(\cdot\,;x_1)-K_1(\cdot\,;x_1)
\geq\Big(-1-\frac{\gamma}{2}\Big)v_1^2(\cdot\,;x_1)
\]
and
\[
\frac{dv_1}{dt}(\cdot\,;x_2)\leq\Big(-1+\frac{\gamma}{2}\Big)v_1^2(\cdot\,;x_2).
\]
We then integrate them over the interval $(t_1,t_2)$ to show that
\[
\hspace*{-15pt}v_1(t_2;x_1)\geq\frac{v_1(t_1;x_1)}{1+(1+\frac{\gamma}{2})v_1(t_1;x_1)(t_2-t_1)}
\quad\text{and}\quad
v_1(t_2;x_2)\leq\frac{v_1(t_1;x_2)}{1+(1-\frac{\gamma}{2})v_1(t_1;x_2)(t_2-t_1)}.
\]
The latter inequality and \eqref{def:x2} imply that 
\[
m(t_2)\leq\frac{m(t_1)}{1+(1-\frac{\gamma}{2})m(t_1)(t_2-t_1)}.
\]
The former inequality and \eqref{def:x1}, on the other hand, imply that
\begin{align*}
v_1(t_2;x_1)>&\frac{(1-\gamma)m(t_1)}{1+(1+\frac{\gamma}{2})(1-\gamma)m(t_1)(t_2-t_1)} \\
>&\frac{(1-\gamma)m(t_1)}{1+(1-\frac{\gamma}{2})m(t_1)(t_2-t_1)} \\
\geq& (1-\gamma)m(t_2).
\end{align*}
A contradiction therefore completes the proof.
\end{proof}

\begin{proof}[Proof of Lemma~\ref{lem:Stirling}]
We use Stirling's inequality and calculate that 
\begin{align*}
\sum_{j=2}^n\left(\begin{matrix} n\\j\end{matrix}\right)&j^{(j-1)/\alpha+1}(n+1-j)^{(n+1-j)/\alpha} \\
\leq&\sum_{j=2}^n\frac{n^n}{j^j(n-j)^{n-j}}j^{(j-1)/\alpha+1}(n+1-j)^{(n+1-j)/\alpha} \\
=&n^{n/\alpha+1}\sum_{j=2}^n\Big(\frac{n+1-j}{n-j}\Big)^{n-j}\frac{n+1-j}{n}
\Big(\frac{j^{j-1}(n+1-j)^{(n+1-j)}}{n^n}\Big)^{1/\alpha-1} \\
\leq&2en^{n/\alpha+1}\sum_{j=2}^{[n/2]}\Big(\frac12\Big)^{(j-1)(1/\alpha-1)}.
\end{align*}
Therefore \eqref{sum:h} follows from \eqref{I:sigma}. In the first inequality we assume the convention $0^0=1$, and in the last inequality $[a]$ denotes the greatest integer not exceeding $a\in\mathbb{R}$. 
\end{proof}

\end{appendix}

\bibliographystyle{amsalpha}
\bibliography{breakingBib}

\providecommand{\bysame}{\leavevmode\hbox to3em{\hrulefill}\thinspace}
\providecommand{\MR}{\relax\ifhmode\unskip\space\fi MR }
\providecommand{\MRhref}[2]{%
  \href{http://www.ams.org/mathscinet-getitem?mr=#1}{#2}
}
\providecommand{\href}[2]{#2}
\begin{thebibliography}{HP16b}

\bibitem[Ami84]{Amick}
Charles~J. Amick, \emph{Regularity and uniqueness of solutions to the
  {B}oussinesq system of equations}, J. Differential Equations \textbf{54}
  (1984), no.~2, 231--247. \MR{757294 (86a:35120)}

\bibitem[BF67]{BF}
T.~B. Benjamin and J.~E. Feir, \emph{The disintegration of wave trains on deep
  water. {P}art 1. {T}heory}, J. Fluid Mech. \textbf{27} (1967), no.~3,
  417--437.

\bibitem[Bou77]{Bnesq1877}
Joseph Boussinesq, \emph{Essai sur la {T}h\'eorie des {E}aux {C}ourantes},
  vol.~23, M\'emoires pr\'esent\'es par diver\'s savants \'a l'{A}cad\'emie des
  {S}ciences l'{I}nstitut de {F}rance (s\'erie 2), no.~1, Paris, Imprimerie
  Nationale, 1877.

\bibitem[CE98]{CE98}
Adrian Constantin and Joachim Escher, \emph{Wave breaking for nonlinear
  nonlocal shallow water equations}, Acta Math. \textbf{181} (1998), no.~2,
  229--243. \MR{1668586 (2000b:35206)}

\bibitem[DO11]{DO}
Bernard Deconinck and Katie Oliveras, \emph{The instability of periodic surface
  gravity waves}, J. Fluid Mech. \textbf{675} (2011), 141--167. \MR{2801039
  (2012j:76057)}

\bibitem[Dob87]{Dobrokhotov}
S.~Yu. Dobrokhotov, \emph{Nonlocal analogues of the nonlinear {B}oussinesq
  equation for surface waves over an uneven bottom and their asymptotic
  solutions}, Dokl. Akad. Nauk SSSR \textbf{292} (1987), no.~1, 63--67.
  \MR{871954 (88a:35202)}

\bibitem[HJ15]{HJ2}
Vera~Mikyoung Hur and Mathew~A. Johnson, \emph{Modulational instability in the
  {W}hitham equation for water waves}, Stud. Appl. Math. \textbf{134} (2015),
  no.~1, 120--143. \MR{3298879}

\bibitem[HP16a]{HP2}
Vera~Mikyoung Hur and Ashish~Kumar Pandey, \emph{Modulational instability in a
  full-dispersion shallow water model}, arxiv:1608.04685 (2016).

\bibitem[HP16b]{HP1}
Vera~Mikyoung Hur and Ashish~Kumar Pandey, \emph{Modulational instability in
  nonlinear nonlocal equations of regularized long wave type}, Physica D:
  Nonlinear Phenomena \textbf{325} (2016), 98 -- 112.

\bibitem[HT14]{HT1}
Vera~Mikyoung Hur and Lizheng Tao, \emph{Wave breaking for the {W}hitham
  equation with fractional dispersion}, Nonlinearity \textbf{27} (2014),
  no.~12, 2937--2949. \MR{3291137}

\bibitem[Hur15]{Hur-breaking}
Vera~Mikyoung Hur, \emph{Breaking in the {W}hitham equation for shallow water
  waves}, arxiv:1506.04075 (2015).

\bibitem[Kat83]{Kato}
Tosio Kato, \emph{On the {C}auchy problem for the (generalized) {K}orteweg-de
  {V}ries equation}, Studies in applied mathematics, Adv. Math. Suppl. Stud.,
  vol.~8, Academic Press, New York, 1983, pp.~93--128. \MR{759907 (86f:35160)}

\bibitem[Lan13]{Lannes}
David Lannes, \emph{The water waves problem}, Mathematical Surveys and
  Monographs, vol. 188, American Mathematical Society, Providence, RI, 2013,
  Mathematical analysis and asymptotics. \MR{3060183}

\bibitem[NS94]{NS94}
P.~I. Naumkin and I.~A. Shishmar{\"e}v, \emph{Nonlinear nonlocal equations in
  the theory of waves}, Translations of Mathematical Monographs, vol. 133,
  American Mathematical Society, Providence, RI, 1994, Translated from the
  Russian manuscript by Boris Gommerstadt. \MR{1261868 (94m:35230)}

\bibitem[Sau15]{Saut}
Jean-Claude Saut, \emph{Private communications}, 2015.

\bibitem[Sch81]{Schonbek}
Maria~Elena Schonbek, \emph{Existence of solutions for the {B}oussinesq system
  of equations}, J. Differential Equations \textbf{42} (1981), no.~3, 325--352.
  \MR{639225 (83b:35151)}

\bibitem[Sel68]{Seliger}
R.~L. Seliger, \emph{A note on the breaking of waves}, Proc. R. Soc. Lond. Ser.
  A Math. Phys. Eng. Sci (1968), 493--496.

\bibitem[Whi67]{Whitham1967}
G.~B. Whitham, \emph{Non-linear dispersion of water waves}, J. Fluid Mech.
  \textbf{27} (1967), 399--412. \MR{0208903 (34 \#8711)}

\bibitem[Whi74]{Whitham}
Gerald~B. Whitham, \emph{Linear and nonlinear waves}, Wiley-Interscience [John
  Wiley \& Sons], New York-London-Sydney, 1974, Pure and Applied Mathematics.
  \MR{0483954 (58 \#3905)}

\bibitem[Yos82]{Yosihara}
Hideaki Yosihara, \emph{Gravity waves on the free surface of an incompressible
  perfect fluid of finite depth}, Publ. Res. Inst. Math. Sci. \textbf{18}
  (1982), no.~1, 49--96. \MR{660822 (83k:76017)}

\end{thebibliography}

\end{document}